\documentclass[11pt]{article}

\usepackage[a4paper, left=2.5cm,right=2.5cm, top=2.5cm, bottom =3cm]{geometry} 

\usepackage{amsmath,amssymb,amsthm, comment, array, colonequals}
\usepackage{parskip}
\usepackage{multirow}
\usepackage[all]{xy}

\usepackage[cal=boondoxo]{mathalpha} %For the fancy lowercase u_i
\DeclareMathAlphabet{\mathcalcm}{OMS}{cmsy}{m}{n}
\newcommand{\cI}{\mathcalcm{I}}

\usepackage{setspace}
\usepackage{mathtools}
\usepackage{tikz-cd}
\usetikzlibrary{calc}

\usepackage[colorlinks=true, linkcolor=blue, citecolor=blue, urlcolor=blue,
pagebackref=true]{hyperref}

\newtheorem{theorem}{Theorem}
\newtheorem{proposition}[theorem]{Proposition}
\newtheorem{corollary}[theorem]{Corollary}
\newtheorem{lemma}[theorem]{Lemma}

\newtheorem{conjecture}[theorem]{Conjecture}
\newtheorem*{definition*}{Definition}

\theoremstyle{remark}
\newtheorem{remark}[theorem]{Remark}
\newtheorem{example}[theorem]{Example}
\newtheorem{question}[theorem]{Question}

\def\A{\mathbb{A}}
\def\P{\mathbb{P}}
\def\Z{\mathbb{Z}}

\def\R{\mathbb{R}}
\def\C{\mathbb{C}}
\def\Q{\mathbb{Q}}

\def\F{\mathbb{F}}

\DeclareMathOperator{\Hom}{Hom}
\DeclareMathOperator{\Spec}{Spec}
\DeclareMathOperator{\SL}{SL}
\DeclareMathOperator{\GL}{GL}
\DeclareMathOperator{\SO}{SO}
\DeclareMathOperator{\OO}{O}
\DeclareMathOperator{\Sp}{Sp}

\DeclareMathOperator{\rank}{rank}
\DeclareMathOperator{\spn}{span}
\DeclareMathOperator{\diag}{diag}

\DeclareMathOperator{\Alt}{Alt}
\DeclareMathOperator{\End}{End}

\def\l({\left(}
\def\r){\right)}

\newcommand{\defi}[1]{\textsf{#1}} % for defined terms

\newcommand{\Gr}{\operatorname{Gr}}

\begin{document}

\title{Sharp bounds for frame counts and setwise stabilizers in classical groups}

\author{Kaloyan Slavov}

\maketitle

\begingroup
\renewcommand{\thefootnote}{}
\footnotetext{%
\\
\makebox[1.8em]{}\textit{Key words and phrases.} frames, classical algebraic groups, stabilizers, finite sets, Brascamp--Lieb inequalities.}
\endgroup

\begin{abstract}
Let $k$ be a field, let $V=k^n$, and let $G$ be a classical group acting on $V$. For a finite subset $E\subset V$ and a basis $u=(u_1,\dots,u_n)$ of $V$, we study the set of $G$-frames of type $u$ contained in $E$, or, equivalently, the set $T_{E,u}^G:=\{g\in G(k)\ |\ g u_i\in E\text{ for each $i$}\}$. In the cases described below, we prove estimates of the form
$|T_{E,u}^G|\ll_n |E|^\alpha$ that are uniform over all fields, with sharp exponents in this uniform setting. For $G=\SL_n$, the uniform sharp exponent is $n-1/n$. For orthogonal groups in dimensions $n=2,3$, with $\operatorname{char}(k)\neq 2$, the uniform sharp exponent is $n/2$.  

We propose an
algebro-geometric Brascamp--Lieb inequality which would lead to the
orthogonal exponent $n/2$ in all dimensions. 

For the related setwise stabilizer
$R_E^G$ of $E$ in $G(k)$, where
$E\subset V$ is finite and spans
$V$, we also prove the sharp characteristic-zero bound
$|R_E^G|\ll_n |E|^{\operatorname{rank}G}$ for the
special linear, orthogonal, and symplectic groups, where
$\operatorname{rank}G$ denotes the absolute rank.
\end{abstract}

\section{Introduction}

Let $k$ be a field, and let $V=k^n$.  We consider $G$ among the following classical groups acting on $V$:
\[
\GL_n,\qquad \SL_n,\qquad \OO(V,B),\qquad \SO(V,B),\qquad \Sp(V,\omega)
\]
(the orthogonal groups are defined when $\operatorname{char}(k)\neq 2$ with respect to a non-degenerate symmetric bilinear form $B$ on $V$, and the symplectic group is defined when $n$ is even with respect to a non-degenerate alternating bilinear form $\omega$). 

Let $u=(u_1,\dots,u_n)$ be an ordered basis of $V$. 
A \defi{$G$-frame} of type $u$ is an ordered basis $(x_1,\dots,x_n)$ of $V$ such that there exists a $g\in G(k)$ with $x_i=gu_i$ for all $i=1,\dots, n$. 

For example, an $\SL_n$-frame of type $u$ is an ordered tuple $(x_1,\dots,x_n)$ of elements in $V$ such that $\det(x_1,\dots,x_n)=\tau$, where $\tau=\det(u_1,\dots,u_n)$. An $\OO(V,B)$-frame is an ordered tuple $(x_1,\dots,x_n)$ of elements in $V$ whose Gram matrix $B(x_i,x_j)_{1\leq i,j\leq n}$ equals the Gram matrix $B(u_i,u_j)_{1\leq i,j\leq n}$ of $u_1,\dots,u_n$. 

Given a finite subset $E\subset V$ and $G$ among the classical groups listed above, we are interested in $G$-frames $(x_1,\dots,x_n)$ of type $u$ with $x_i\in E$ for every $i=1,\dots,n$. The set of such frames is in bijection with 
\[T_{E,u}^G:=\{g\in G(k)\ |\ g u_i\in E\text{ for each $i$}\}.\]
There is an injective map
$T_{E,u}^G \longrightarrow E^n$ given by
$g \longmapsto (g u_1,\dots,g u_n)$; therefore,
$T_{E,u}^G$ is finite, and we have the trivial bound
\begin{equation}
	|T_{E,u}^G|\leq |E|^n.
	\label{eq:trivial_bound}
\end{equation}

In this paper, $X\ll Y$ means $X\leq CY$ for a constant $C$; a subscript indicates what the constant is allowed to depend on. Define $X\gg Y$ similarly. Also, $X\asymp Y$ means $X\ll Y$ and $X\gg Y$. 

We ask the following
\begin{question}
	\label{Que:size_frame}
How large can $|T_{E,u}^G|$ be in terms of $|E|$?
Concretely, what is the smallest exponent $\alpha$ such that 
\[|T_{E,u}^G|\ll_n |E|^\alpha?\] In other words, how much power saving from the trivial bound \eqref{eq:trivial_bound} can we gain? 
\end{question}

The motivation for the current paper is a recent work of  Pham, Le Quang-Hung and Slavov \cite{Vietnam_BLMS}. For a finite field $k$ and a subset $E\subset k^2$ that spans $k^2$, they bound the size of the setwise stabilizer $R_E$ of $E$ in $\SL_2(k)$. Namely, they establish $|R_E|\ll |E|^{3/2}$, with the exponent $3/2$ being sharp. 
They ask for higher-dimensional analogues.

The present paper addresses this question in a broader setting. We work over an arbitrary field, treat other classical groups as well, and pass from setwise stabilizers to the associated frame-counting problem described above. The comparison between the two problems is recorded in
\eqref{eq:comparison_R_E_T_E}. See also Remark \ref{Rem:comparison_R_E_versus_T_E_level_sets} for the comparison explicitly in the case of $\SL_n$.

For $G=\GL_n$, the bound (\ref{eq:trivial_bound}) is sharp (in every characteristic) up to a multiplicative constant depending only on $n$. See Example \ref{Exa:sharp_GL_n}. All examples in the paper are assembled in Section \ref{Sec:examples}. However, the exponent in \eqref{eq:trivial_bound} can be improved for the other groups that we treat. 

The table below summarizes the exponents that arise in the present paper.
In each row, the listed exponent cannot be
replaced by any smaller exponent in a bound valid across all fields under
consideration. In the upper-bound column, $\checkmark$ means that the upper
bound for $T_{E,u}^G$ with the indicated exponent is proved unconditionally, $\dagger$ means that it is conditional
on Conjecture \ref{Conj:Brascamp_Lieb}, and $?$ means that it is
conjectural. We always assume $\operatorname{char}(k)\neq 2$ when
$G$ is $\OO(V,B)$ or $\SO(V,B)$.

\begin{table}[htbp]
\centering
	\begingroup
	\renewcommand{\arraystretch}{1.35}
	\begin{tabular}{|c|c|c|c|}
		\hline
		$G$
		&
	exponent
		&
		upper bound status
		&
lower-bound	example
		\\ \hline
		
		$\mathrm{GL}_n$
		& $n$
		& $\checkmark$ trivial bound \eqref{eq:trivial_bound}
		& Example \ref{Exa:sharp_GL_n}
		\\ \hline
		
		$\mathrm{SL}_n$
		& $n-\dfrac{1}{n}$
		& $\checkmark$ Theorem \ref{Thm:main_high_dimensions}
		& Example \ref{Exa:positive_char_SL_n}
		\\ \hline
		
		$\begin{array}{c}
			\mathrm{O}(V,B)\\
			\mathrm{SO}(V,B)
		\end{array}$
		& $n/2$
		&
		\begin{tabular}{@{}l@{}}
			$\checkmark$ $n=2,3$, Theorem \ref{Thm_orthogonal_cases_n_2_3}\\
			$\dagger$ $n\geq 4$, Corollary \ref{Cor:General_case_orthogonal}
		\end{tabular}
		& Example \ref{Exa:positive_char_SO_n}
		\\ \hline
		
		$\mathrm{Sp}(V,\omega)$
		& $3n/4$
		&
		\begin{tabular}{@{}l@{}}
			$\checkmark$ $n=2$, Theorem \ref{Thm:main_high_dimensions}\\
			$?$ $n\geq 4$, Conjecture \ref{Conj_symplectic}
		\end{tabular}
		& Example \ref{Exa:sharpness_symplectic}
		\\ \hline
	\end{tabular}
	\endgroup
	\caption{Exponents forced by examples and status of the corresponding upper bounds}
		\label{Tab:summary_results}	
\end{table}

\begin{remark}
One can read the entries in the above table in terms of the power savings from the exponent $n$ in the trivial bound \eqref{eq:trivial_bound}. For example, in the case of $\SL_n$, the power saving is only $1/n$, while in the case of the orthogonal groups,
there is a substantial power saving of $n/2$.   
\end{remark}

\subsection{The special linear group}
\label{Sec:Intro_SL_n}

The result below gives the upper bound in the case $G=\SL_n$. 

\begin{theorem}
\label{Thm:main_high_dimensions}
Let $k$ be a field. Let $E\subset k^n$ be a finite subset. Let $u=(u_1,\dots,u_n)$ be a basis of $V$. Then
\[
|T_{E,u}^{\SL_n}|\ll_n |E|^{n-1/n}.
\]
\end{theorem}

Setting $\tau:=\det(u_1,\dots,u_n)$, Theorem \ref{Thm:main_high_dimensions} is equivalent to the following

\begin{theorem}\label{Thm:SL_n_statement_frames}
Let $k$ be a field. Let $E\subset k^n$ be a finite subset. Let $\tau\in k^*$. Then
\[\#\{(x_1,\dots,x_n)\in E^n\ |\ \det(x_1,\dots,x_n)=\tau\}\ll_n |E|^{n-1/n}.\]	
\end{theorem}

See Section \ref{Sec:SL_n} for the proof.  

Example \ref{Exa:positive_char_SL_n} shows that the exponent $n-1/n$ is sharp, in the sense that it cannot be replaced by any smaller exponent in a bound valid across all fields. In fact, the example shows that the exponent is sharp for every fixed algebraically closed field of positive characteristic. However, the exponent is not sharp when $k=\R$; see \cite[Section~5.2]{Do}.

\subsection{Orthogonal groups in dimensions $2$ and $3$} 
\label{Sec:Intro_orthogonal_n_2_3}

In this section $\operatorname{char} k\ne 2$ and $\OO(V,B)$ is the orthogonal group of a fixed non-degenerate symmetric bilinear form $B$ on $V$. 

The following result treats the cases $n=2,3$. 

\begin{theorem}\label{Thm_orthogonal_cases_n_2_3}
Let $n\in\{2,3\}$. 	
Let $k$ be a field with $\operatorname{char}(k)\neq 2$, and let $V=k^n$. Let $B$ be a non-degenerate symmetric bilinear form on $V$. Let $G$ be one of the groups $\OO(V,B)$ or $\SO(V,B)$. Let $u=(u_1,\dots,u_n)$ be a basis of $V$, and let $E\subset k^n$ be a finite subset. Then
\[
|T_{E,u}^{G}|\ll_n |E|^{n/2}.
\] 
\end{theorem}

See Section \ref{Sec:orthogonal_n_2_3} for the proof. 

Example \ref{Exa:positive_char_SO_n} shows that the exponent $n/2$ cannot be
replaced by any smaller exponent in a bound valid across all fields of characteristic different from $2$. In fact, the example shows that the exponent is sharp over every fixed algebraically closed field of positive characteristic different from $2$. 

\subsection{Orthogonal groups: a Brascamp--Lieb approach}

We describe a possible route to the bound in the orthogonal case in arbitrary dimension. Let $G=\OO(V,B)$.
For each $i=1,\dots,n$, consider the map
\[\pi_i:G\to \A^n,\quad g\mapsto gu_i.\]
The set $T_{E,u}^G$ consists precisely of all $g\in G(k)$ such that $\pi_i(g)\in E$ for all $i$. Thus the bound we are trying to establish takes the form 
\begin{equation}\label{eq:O_n_goal_bound}
	\#\{g\in G(k)\ |\ \pi_i(g)\in E\text{\ for each $i$}\}\ll_{n,k} |E|^{n/2}
\end{equation}
(we allow dependence on $k$ here as well; see Remark \ref{Rem:Orthogonal_BL_dependence_on_k}). 

We now state the classical Brascamp--Lieb inequality from harmonic analysis and then formulate an algebro-geometric version which, if established, would allow us, after a verification specific to the orthogonal group, to deduce \eqref{eq:O_n_goal_bound}.

Let $H,H_1,\dots,H_m$ be finite-dimensional real Hilbert spaces, 
with induced Lebesgue measures respectively
$dx, dy_1,\dots dy_m$, and let $\pi_j:H\to H_j$ be linear maps ($j=1,\dots,m$). Let $p_1,\dots,p_m>0$. The Brascamp--Lieb inequality states that
under certain assumptions, there exists a constant $C>0$ such that for all measurable functions $f_j:H_j\to\R_{\geq 0}$ with $\int_{H_j}f_j(y_j)\, dy_j<\infty$, the following holds:
\begin{equation}
\label{eq:BL_integrals}	
\int_H \prod_{j=1}^m (f_j\circ\pi_j (x))^{p_j}\, dx\leq C\prod_{j=1}^m\left(\int_{H_j}f_j(y_j)\,dy_j\right)^{p_j}.
\end{equation}

Bennett, Carbery, Christ, and Tao prove \cite[Theorem~1.15]{BCCT} that under a natural scaling condition, \eqref{eq:BL_integrals} holds provided that for every subspace $V\subset H$, 
\begin{equation}
\label{eq:BL_subspaces_V_of_Hilb_space}
\dim V\leq \sum_{j=1}^m p_j\dim\pi_j(V).
\end{equation}	 

In \cite{BCCT2}, the same authors also establish the following discrete version for finitely generated abelian groups. The $p_j$ in the statement below correspond to $1/p_j$ in \cite{BCCT2}, and the $f_j$ below correspond to $f_j^{p_j}$ from \cite{BCCT2}. 

\begin{theorem}[Theorem~2.4 in \cite{BCCT2}]
Let $G,G_1,\dots,G_m$ be finitely generated abelian groups, and let
$\pi_j:G\to G_j$ be group homomorphisms ($j=1,\dots,m$). Let $p_1,\dots,p_m\in (0,1]$. Suppose that
\begin{equation}
	\label{eq:BL_ab_subgroups_condition}
\rank H\leq \sum_{j=1}^m p_j \rank \pi_j(H)\quad\text{for all subgroups $H\subset G$}.
\end{equation} 
Then there exists a $C>0$ such that for all functions $f_j: G_j\to \R_{\geq 0}$ with $\sum_{y_j\in G_j} f_j(y_j)<\infty$, the following holds: 
\begin{equation}
\label{eq:BL_discrete_ab_gps_conclusion}	
\sum_{x\in G}\prod_{j=1}^m (f_j\circ\pi_j(x))^{p_j}\leq C \prod_{j=1}^m \left(\sum_{y_j\in G_j} f_j(y_j)\right)^{p_j}.
\end{equation}
In particular, taking the $f_j$'s to be the indicator functions of finite subsets $A_j\subset G_j$, 
\begin{equation*}
\#\{x\in G\ |\ \pi_j(x)\in A_j\text{ for all }j\}
\leq C\prod_{j=1}^m |A_j|^{p_j}.
\end{equation*}
\end{theorem}

An analogous statement has been established in \cite{mult_linear_Brascamp_Lieb} in the context of finite-dimensional vector spaces over a field. The $p_j$ below correspond to the $s_j$ in \cite{mult_linear_Brascamp_Lieb}, and the $f_j$ below correspond to the $f_j^{1/s_j}$ from \cite{mult_linear_Brascamp_Lieb}. 

\begin{theorem}[Theorem~2.5 in \cite{mult_linear_Brascamp_Lieb}]	
Let $k$ be a field, let $V, V_1,\dots,V_m$ be finite-dimensional vector spaces over $k$, and let $\pi_j:V\to V_j$ be linear maps ($j=1,\dots,m$). Let $p_1,\dots,p_m\in(0,1]$. Suppose that
\begin{equation}
	\label{eq:BL_v_spaces_subspace_condition}
	\dim L\leq \sum_{j=1}^m p_j \dim \pi_j(L)\quad\text{for all subspaces $L\subset V$}.
\end{equation} 
Then there exists a $C>0$ such that for all functions $f_j: V_j\to \R_{\geq 0}$ with $\sum_{y_j\in V_j}f_j(y_j)<\infty$, the following holds: 
\begin{equation*}
	\sum_{x\in V}\prod_{j=1}^m (f_j\circ\pi_j(x))^{p_j}\leq C \prod_{j=1}^m \left(\sum_{y_j\in V_j} f_j(y_j)\right)^{p_j}.
\end{equation*}
In particular, taking the $f_j$'s to be the indicator functions of finite subsets $A_j\subset V_j$, 
\begin{equation*}
	\#\{x\in V\ |\ \pi_j(x)\in A_j\text{ for all }j\}
	\leq C\prod_{j=1}^m |A_j|^{p_j}.
\end{equation*}
\end{theorem}

Motivated by these, we now propose an algebro-geometric version, whereby Hilbert spaces (respectively, finitely generated abelian groups, respectively, vector spaces) are replaced by algebraic varieties, linear maps (respectively, group homomorphisms) by morphisms, and subspaces by irreducible subvarieties.

\begin{conjecture}\label{Conj:Brascamp_Lieb}	
Let $k$ be a field, let $X, Y_1, \dots, Y_m$ be affine varieties\footnote{A variety over $k$ is a scheme of finite type over $k$.} over $k$, and let  $\pi_j: X\to Y_j$ be morphisms ($j=1,\dots,m$). Let $p_1,\dots,p_m\in (0,1]$. Suppose that\footnote{Here $(\pi_j)_{\overline{k}}$ denotes base change to $\overline{k}$; then we take closure of the image.}
\begin{equation}
\dim Z\leq \sum_{j=1}^m p_j\dim \overline{(\pi_{j})_{\overline{k}}(Z)}\quad\text{for every irreducible subvariety $Z\subset X_{\overline{k}}$}.
\label{eq:dimension_subvariety_Brascamp_Lieb}
\end{equation}
Then there exists a $C>0$ such that for all finite subsets $A_j\subset Y_j(k)$, the following holds:
\begin{equation}
\#\{x\in X(k)\ |\ \pi_j(x)\in A_j\text{ for all }j\}
\leq C\prod_{j=1}^m |A_j|^{p_j}.
\label{eq:Brascamp_Lieb_conclusion}
\end{equation}
\end{conjecture}

We can also state a stronger version that takes as input finitely supported functions $f_j: Y_j(k)\to\R_{\geq 0}$ and replaces \eqref{eq:Brascamp_Lieb_conclusion}
more generally by the analogue of \eqref{eq:BL_discrete_ab_gps_conclusion}; then \eqref{eq:Brascamp_Lieb_conclusion} appears when $f_j$ are specialized to the indicator functions of the sets $A_j$. 

Examples \ref{Exa:failure_BL_proj_axis}, \ref{Exa:failure_BL_one_half}, and \ref{Exa:failure_BL_cuspidal} show that one cannot discard the condition for subvarieties $Z$ and that it is not enough to check (\ref{eq:dimension_subvariety_Brascamp_Lieb}) for irreducible components $Z$ of $X$. 

\begin{remark}
Duncan \cite{Duncan} proves a weighted Brascamp--Lieb
inequality in a Euclidean real-variable setting, with source an open subset of a
real algebraic variety and targets Riemannian manifolds, for a
class of maps that are $C^\infty$ on an open dense subset of their
domain. This class contains, and is substantially broader than, the
class of polynomial maps over $\mathbb R$. The inequality in \cite{Duncan} is weighted and involves no hypothesis analogous to  \eqref{eq:dimension_subvariety_Brascamp_Lieb} imposed on the images of irreducible subvarieties. Thus, even when $k=\mathbb R$, the two statements address different types of Brascamp--Lieb inequalities despite the similar terminology.
\end{remark}

\begin{remark}
A recent work of Johnsrude \cite[Theorem~9.4]{Johnsrude}
relates to the case when all of $X$, $Y_1, \dots, Y_m$ are affine spaces. Namely,
Johnsrude studies a setting with polynomial maps $P_j: \Z^n\to \Z^{n_j}$ with integer coefficients, $j=1,\dots,m$, and establishes a Brascamp--Lieb inequality under a scaling condition and the assumption that, 
at every point $z\in\C^n$, the differentials
$dP_j(z)$ are surjective and satisfy the dimension inequalities
\eqref{eq:BL_v_spaces_subspace_condition},  tested against every complex-linear
subspace of $\C^n$.
\end{remark}

In Section \ref{Sec:Brascamp_Lieb}, we assume Conjecture \ref{Conj:Brascamp_Lieb} and deduce the following conditional corollary.  

\begin{corollary}\label{Cor:General_case_orthogonal}
Assume Conjecture \ref{Conj:Brascamp_Lieb}. 
Let $k$ be a field with $\operatorname{char}(k)\neq 2$, and let $V=k^n$. Let $B$ be a non-degenerate symmetric bilinear form on $V$. Let $G$ be one of the groups $\OO(V,B)$ or $\SO(V,B)$. Let $E\subset V$ be a finite subset, and let $u=(u_1,\dots,u_n)$ be a basis of $V$. 
Then
\begin{equation}
	\label{eq:bound_cor_orthogonal}
|T_{E,u}^G|\ll_{n,k} |E|^{n/2}.
\end{equation}
\end{corollary}

For $n\geq2$, Example \ref{Exa:positive_char_SO_n} 
shows that the exponent $n/2$ cannot be replaced by any smaller exponent, both in a bound valid across all
fields of characteristic different from $2$ and in a bound over any fixed algebraically closed field of positive characteristic different from $2$.

The key step in the proof is verifying that the orthogonal group $\OO(V,B)$ together with the maps 
$\pi_j:\OO(V,B)\to \A^n$, $g\mapsto gu_j$
satisfies the Brascamp--Lieb condition \eqref{eq:dimension_subvariety_Brascamp_Lieb} with $p_j=1/2$ for $j=1,\dots,n$. See Proposition \ref{Prop:OO_satisfies_Brascamp_Lieb}. 

\begin{remark}
	\label{Rem:Orthogonal_BL_dependence_on_k}
While the implied constant in \eqref{eq:bound_cor_orthogonal} is allowed to depend on the field $k$, in Section \ref{Sec:BL_Alternative_versions} we discuss alternative versions of Conjecture \ref{Conj:Brascamp_Lieb} that would eliminate that dependence. 
\end{remark}

\begin{remark}
When $k=\R$, the group $G$ is the standard $\OO(n)$ or $\SO(n)$, and the basis $u_1,\dots,u_n$ is orthonormal, the conclusion of Corollary \ref{Cor:General_case_orthogonal} holds unconditionally --- this follows from \cite[Proposition 12]{BCELM}. See Section \ref{Sec:real_orthogonal_groups}. This is further evidence for Conjecture \ref{Conj:Brascamp_Lieb}.  	
\end{remark}

\subsection{A symplectic conjecture} 

The exponent $3n/4$ in the conjecture below arises from Example \ref{Exa:sharpness_symplectic}, which shows that it cannot be replaced by a smaller exponent in a bound valid across all fields, or in a bound over any fixed algebraically closed field of positive characteristic.

\begin{conjecture}
	\label{Conj_symplectic}
Let $k$ be a field, and let $V=k^n$. Assume that $n$ is even, and let $\omega$ be a non-degenerate alternating form on $V$. Consider the symplectic group $G=\Sp(V,\omega)$. Let $E\subset V$ be a finite subset, and let $u=(u_1,\dots,u_n)$ be a basis of $V$. Then
\[
|T_{E,u}^{G}|\ll_n |E|^{3n/4}.
\] 
\end{conjecture}

The case $n=2$ is included in Theorem \ref{Thm:main_high_dimensions}, since
$\Sp_2\simeq \SL_2$.  In higher dimensions, the conjecture is motivated by
the pattern established for the other classical groups: the predicted
exponent is the one arising from finite-subfield examples. 

\subsection{Setwise stabilizers}
\label{Sec:intro_setwise_stabilizers}

For a finite subset $E\subset V$ that spans $V$, a related object\footnote{Notation as in \cite{Vietnam_BLMS}.} (representing the original motivation for the current work) is the setwise stabilizer
\[
R_E^G=\{g\in G(k)\ |\ g(E)=E\}.
\]
Taking a basis $u_1,\dots,u_n$ of $V$ with all $u_i\in E$ gives
\begin{equation}
	\label{eq:comparison_R_E_T_E}
	R_E^G\subset T_{E,u}^G,	\quad\text{so in particular}\quad |R_E^G|\leq |T_{E,u}^G|.
\end{equation}	

Therefore all upper bounds established for $|T_{E,u}^G|$ apply also to $|R_E^G|$.  Moreover, in every lower-bound example cited in Table \ref{Tab:summary_results}, the exponent listed in Table \ref{Tab:summary_results} is
already attained by the stabilizer $R_E^G$, rather than only by $T_{E,u}^G$. 
Consequently, Table \ref{Tab:summary_results} yields corresponding
established, conditional, and conjectural statements for $R_E^G$,
with the same status. It is perhaps surprising that, even though
$T_{E,u}^G$ is in general larger and less rigid than $R_E^G$ for a given $E$, their extremal growth exponents are the same in all cases where they have been determined. This agreement holds both when the field is allowed to vary and over every fixed algebraically closed field of positive characteristic (different from $2$ in the orthogonal cases).

For $G=\SL_n$, the stabilizer bound
$|R_E^{\SL_n}|\ll_n |E|^{n-1/n}$
also admits a short direct proof, independent of Theorem
\ref{Thm:main_high_dimensions}. We include it in Section
\ref{Sec:short_proof_R_E_special_linear}.

The stabilizer problem admits a clean answer not only in the positive-characteristic cases already described, but also in characteristic zero.

\begin{theorem}	\label{Thm:bound_char_0_rank}
	Let $k$ be a field of characteristic zero, and let $V=k^n$. Let $G$ be one of
	\[
	\GL_n,\quad \SL_n,\quad \OO(V,B),\quad \SO(V,B),\quad \Sp(V,\omega)
	\]
	(the orthogonal groups are defined with respect to a non-degenerate symmetric bilinear form $B$ on $V$, and the symplectic group is defined when $n$ is even with respect to a non-degenerate alternating bilinear form $\omega$ on $V$). Let $E\subset V$ be a finite subset that spans $V$.
	Then
	\begin{equation}
		|R_E^G|\ll_n |E|^{\rank G},
		\label{eq:bound_rank}
	\end{equation}
	where $\rank G$ denotes absolute rank\footnote{See Table \ref{Tab:classical-groups-ranks}.}. 
\end{theorem}

\medskip

\begin{table}[htbp]	
	\centering
	\begingroup
	\renewcommand{\arraystretch}{1.25}
	\begin{tabular}{|c|c|}
		\hline
		\(G\) & \(\operatorname{rank} G\) \\
		\hline
		\(\GL_n\) & \(n\) \\
		\hline
		\(\SL_n\) & \(n-1\) \\
		\hline
		\(\OO(V,B)\) & \(\left\lfloor n/2 \right\rfloor\) \\
		\hline
		\(\SO(V,B)\) & \(\left\lfloor n/2 \right\rfloor\) \\
		\hline
		\(\Sp(V,\omega)\) & \(n/2\) \\
		\hline
	\end{tabular}
	\endgroup
	\caption{Absolute ranks of the classical groups}
\label{Tab:classical-groups-ranks}	
\end{table}

\medskip

The proof is in Section \ref{Sec:char_0}.

For each of the groups $G$ considered here, the exponent 
$\rank G$ in (\ref{eq:bound_rank}) is sharp (in the sense that it cannot be replaced by a smaller exponent in the statement of Theorem \ref{Thm:bound_char_0_rank}) ; see Example \ref{Exa:sharp_char_0_rank}. 

\begin{remark}
One can compare the exponents for $|R_E^G|$ listed in Table
\ref{Tab:summary_results} with the (sharp) exponents for $|R_E^G|$ in characteristic zero, listed in Table \ref{Tab:classical-groups-ranks}. 
The extra power saving due to restricting to characteristic $0$ is $1-1/n$ for $\SL_n$, and $1/2$ for orthogonal groups when $n$ is odd. 
\end{remark}

\section{Examples}
\label{Sec:examples}

\subsection{Lower-bound examples for setwise stabilizers}

Recall the definition of $R_E^G$ from Section \ref{Sec:intro_setwise_stabilizers}. By \eqref{eq:comparison_R_E_T_E}, lower bounds for $R_E^G$ yield the corresponding lower bounds also for $T_{E,u}^G$.  

\begin{example}
	\label{Exa:sharp_GL_n}
	Let $G=\GL_n$. Let $k$ be an algebraically closed field, so for infinitely many $N$, the group $\mu_N=\{\zeta\in k^*\ |\ \zeta^N=1\}$ of $N$-th roots of unity in $k^*$ has order $N$. 
	Let $v_1,\dots,v_n$ be the standard basis of $V$.  Define
	\begin{equation}
		E_N=\bigcup_{i=1}^n \mu_N v_i.
		\label{eq:E_N_definition}
	\end{equation}
	Then $E_N$ spans $V$ and $|E_N|=nN$.  The finite diagonal group
	\[
	D_N:=\{\diag(\zeta_1,\dots,\zeta_n)\ |\  \zeta_i\in \mu_N\}\subset\GL_n(k)
	\]
	has order $N^n$ and stabilizes $E_N$.  Thus
	\[|R_{E_N}^{\GL_n}|\geq |D_N|=N^n=n^{-n}|E_N|^n\gg_n |E_N|^n.\] 
\end{example}

Examples \ref{Exa:positive_char_SL_n}, \ref{Exa:positive_char_SO_n}, and 
\ref{Exa:sharpness_symplectic} are inspired by \cite[Example~4]{Vietnam_BLMS}. Throughout the paper, whenever we need a basis for an orthogonal space, we follow \cite[Example~2.40]{Moonen} and work with a split basis (Example \ref{Exa:positive_char_SO_n}, Example \ref{Exa:sharp_char_0_rank}, and Section \ref{Subsec:all_principal_minors_vanish}).

\begin{example}	\label{Exa:positive_char_SL_n}
Let $G=\SL_n$. 
Let $k$ be an algebraically closed field of positive characteristic. Consider $\F_{q}\subset k$ and set
	$E=\F_{q}^n\subset V$. Then $E$ spans $V$ and $|E|=q^n$.  Moreover
	$\SL_n(\F_{q})\subset R_E^{\SL_n}$. Therefore
	\[|R_E^{\SL_n}|\geq |\SL_n(\F_q)|\asymp q^{n^2-1}=|E|^{(n^2-1)/n}=|E|^{n-1/n}. \]
\end{example}

\begin{example}
	\label{Exa:positive_char_SO_n}
Let $k$ be an algebraically closed field of positive characteristic $\neq 2$, and let $V=k^n$ with $n\geq 2$.  Let $B$ be any non-degenerate symmetric bilinear form on $V$.  Since $k$ is algebraically closed with $\operatorname{char}(k)\neq 2$, any non-degenerate symmetric bilinear form on $V$ is equivalent to $(x,y)\mapsto x^T y$ (see, for example, \cite[Theorem~8.31]{StollLinearAlgebraII}); in particular, any two non-degenerate symmetric bilinear forms on $V$ are equivalent.  Since the split form described below is also non-degenerate, after choosing a suitable basis we may write $B$ in split form.  In other words, there exists a basis $e_1,\dots,e_r,f_1,\dots,f_r$	if $n=2r$, and a basis
	$e_1,\dots,e_r,f_1,\dots,f_r,h$	if $n=2r+1$, such that
	$B(e_i,e_j)=B(f_i,f_j)=0$, $B(e_i,f_j)=\delta_{ij}$, and, in the odd-dimensional case, $B(h,h)=1$, $B(h,e_i)=B(h,f_i)=0$.
	
	Let $Q(v)=B(v,v)$ be the associated quadratic form. Use the above split basis to identify $V$ with $k^n$. For a finite subfield $\F_q\subset k$, define
	\[
	E:=\{v\in \F_q^n\ |\ Q(v)=0\}\subset V.
	\]
	
	The set $E$ spans $V$.  Indeed, the vectors $e_1,\dots,e_r,f_1,\dots,f_r$ all belong to $E$.  
	If $n$ is odd, the last basis vector $h$ belongs to the span of $E$, because $Q(h+e_1-f_1/2)=1-1=0$ gives $h+e_1-f_1/2\in E$. 
	
	Note that $|E|\asymp_n q^{n-1}$ (with $n=2r$ or $n=2r+1$, count $|E|$ according to whether the coordinates along $e_1,\dots,e_r$ are all zero). The group $\SO(V,B)(\F_q)$ preserves $E$.  Therefore
	\[
	|R_E^{\OO(V,B)}|\geq	|R_E^{\SO(V,B)}|\geq |\SO(V,B)(\F_q)|\asymp_n q^{n(n-1)/2}\asymp_n |E|^{n/2}\]
	(on the second-to-last step, we use the standard formulae for the size of the orthogonal group over finite fields; see, for example, \cite[Chapter~6]{CameronClassicalGroups}).	
\end{example}

\begin{example}
	\label{Exa:sharpness_symplectic}
Let $k$ be an algebraically closed field of positive characteristic.	
	Consider a symplectic space $(V,\omega)$ over $k$. 
	Let $n=2r$. Choose a basis $e_1,\dots,e_r,f_1,\dots,f_r$ of $V$ with $\omega(e_i,e_j)=\omega(f_i,f_j)=0$, $\omega(e_i,f_j)=\delta_{ij}$. Let $V_i=\spn(e_i,f_i)$, so $V=\oplus_{i=1}^r V_i$ (orthogonal direct sum). Moreover, $\omega$ restricts to a non-degenerate alternating form $\omega_i$ on $V_i$. 	
	
	For a finite subfield $\F_q\subset k$, consider 
	$E=\bigcup_{i=1}^r V_i(\F_q)$; here $V_i(\F_q)=\F_q e_i+\F_q f_i$. 
	Note that $E$ spans $V$ and $|E|=r(q^2-1)+1\asymp_n q^2$. The group
	\[G:=\Sp(V_1,\omega_1)(\F_q)\times\dots\times\Sp(V_r,\omega_r)(\F_q)\subset \Sp(V,\omega)\]
	preserves $E$. 
	For each $i$, we have $|\Sp(V_i,\omega_i)(\F_q)|=|\SL_2(\F_q)|\asymp q^3$; thus, $|G|\asymp_n q^{3r}$. Therefore
	\[|R_E^{\Sp(V,\omega)}|\geq |G|\asymp_n q^{3r}\asymp_n |E|^{3r/2}=|E|^{3n/4}.\]
\end{example}

\subsection{Examples for the algebraic Brascamp--Lieb conjecture}

The examples below illustrate that the conclusion \eqref{eq:Brascamp_Lieb_conclusion} of Conjecture \ref{Conj:Brascamp_Lieb} may fail if one does not impose the assumption \eqref{eq:dimension_subvariety_Brascamp_Lieb}. 

\begin{example}
\label{Exa:failure_BL_proj_axis}	
Let $k=\Q$. Let $X=\mathbb A^2_{x,y}$. Consider $\pi_1:X\to \A^2$, $\pi_1(x,y)=(x,xy)$, 	and
$\pi_2:X\to \A^1$, $\pi_2(x,y)=y$. Take $p_1=1$, $p_2=1/2$.
Then (\ref{eq:dimension_subvariety_Brascamp_Lieb}) is satisfied for $Z=X$ but fails for $Z=V(x)$ because $\dim Z=1$ while $\dim\pi_1(Z)=0$ and $\dim\pi_2(Z)=1$. Consider the sets $A_1:=\{(0,0)\}\subset\Q^2$ and $A_2=\{1,\dots, N\}\subset\Q$. The set $E=\{(0,i)\ |\ 1\leq i\leq N\}\subset X(\Q)$ satisfies $\pi_1(E)\subset A_1$ and $\pi_2(E)\subset A_2$. The bound (\ref{eq:Brascamp_Lieb_conclusion}) fails. 
\end{example}

\begin{example}
	\label{Exa:failure_BL_one_half}
Let $k=\Q$. Let
	$X=\mathbb A^2_{x,y}$, and define 
	$\pi_1:X\to \mathbb A^2$, $\pi_1(x,y)=(x,xy)$,
$\pi_2:X\to \mathbb A^2$, $\pi_2(x,y)=(y,xy)$.
Take $p_1=p_2=1/2$. Again, (\ref{eq:dimension_subvariety_Brascamp_Lieb}) is satisfied for $Z=X$. However, it fails for $Z=V(x)$ because 
$\dim\pi_1(Z)=0$ and $\dim\pi_2(Z)=1$. Taking $A_1=\{(0,0)\}\subset\Q^2$,
$A_2=\{(i,0)\ |\ 1\leq i\leq N\}\subset\Q^2$ and $E=\{(0,i)\ |\ 1\leq i\leq N\}\subset X(\Q)$ violates (\ref{eq:Brascamp_Lieb_conclusion}). 
\end{example}

\begin{example}
	\label{Exa:failure_BL_cuspidal}
Let $k=\Q$.	Let $X=\A^2_{x,y}$, and set	$q(x,y)=y^2-x^3$. Define 
	$\pi_1:X\to\A^2$, $\pi_1(x,y)=(q(x,y),xq(x,y))$,
	and $\pi_2:X\xrightarrow{\sim}\A^2$, $\pi_2(x,y)=(x+y,y)$.	Take
	$p_1=p_2=1/2$. As in the previous examples, the dimension condition is satisfied for $Z=X$. However, it fails on $Z=V(y^2-x^3)\subset X$, since $\dim \pi_1(Z)=0$. 
	Moreover, (\ref{eq:Brascamp_Lieb_conclusion}) fails with $A_1=\{(0,0)\}$, $A_2=\{(i^2+i^3, i^3)\ |\ 1\leq i\leq N\}$ because we can take $E=\{(i^2,i^3)\ |\ 1\leq i\leq N\}\subset X(\Q)$. 
\end{example}

\subsection{Examples for the stabilizer problem in characteristic zero}

The example below shows that the exponent $\rank G$ in \eqref{eq:bound_rank} cannot be improved. 

\begin{example}
	\label{Exa:sharp_char_0_rank}
	Let $k$ be an algebraically closed field of characteristic $0$, and let $V=k^n$. Let $\mu_N$ denote the group of $N$-th roots of unity in $k^*$. Let $v_1,\dots,v_n$ be the standard basis of $V$. 
	
	The case $G=\GL_n$ has already been treated in Example \ref{Exa:sharp_GL_n}. We now adapt the construction for the other groups under consideration. 
	
	Let $G=\SL_n$. 		
	Take $E_N$ as in \eqref{eq:E_N_definition} in Example \ref{Exa:sharp_GL_n}; again, $|E_N|=nN$.  
	The subgroup
	\[
	D_N^{\SL_n}=\{\diag(\zeta_1,\dots,\zeta_n)\ |\ \zeta_i\in \mu_N,\ \zeta_1\cdots \zeta_n=1\}\subset\SL_n(k)
	\]
	stabilizes $E_N$ setwise and has order $N^{n-1}$.  We obtain
	\[|R_{E_N}^{\SL_n}|\geq |D_N^{\SL_n}|= N^{n-1}=n^{-(n-1)}|E_N|^{n-1}\gg_n |E_N|^{n-1}.\]
	
	Let $G=\SO(V,B)$. Let $r=\lfloor n/2\rfloor$. As in  Example \ref{Exa:positive_char_SO_n},
	there exists a basis $v_1=e_1,\dots,v_r=e_r, v_{r+1}=f_1,\dots, v_{2r}=f_r$ if $n$ is even and
	$v_1=e_1,\dots,v_r=e_r, v_{r+1}=f_1,\dots, v_{2r}=f_r, v_{2r+1}=h$ if $n$ is odd such that
	$B(e_i,e_j)=B(f_i,f_j)=0$, $B(e_i,f_j)=\delta_{ij}$, and, in the odd-dimensional case, $B(h,h)=1$, $B(h,e_i)=B(h,f_i)=0$.
	
	Take $E_N$ as in \eqref{eq:E_N_definition} in Example \ref{Exa:sharp_GL_n} if $n$ is even, and $E_N=\{h\}\cup\bigcup_{i=1}^{2r}\mu_N v_i$ if $n$ is odd. 
	
	In both cases $E_N$ spans $k^n$ and $|E_N|\leq 2rN+1$.  The subgroup
	\[
	A_N=\{e_i\mapsto \zeta_i e_i,\ f_i\mapsto \zeta_i^{-1}f_i,\ h\mapsto h\text{ if }n\text{ is odd}\ |\ \zeta_i\in \mu_N\}\subset\SO(V,B)
	\]
	stabilizes $E_N$ setwise and has order $N^r$.  Therefore
	\[
	|R_{E_N}^{\OO(V,B)}|\geq |R_{E_N}^{\SO(V,B)}|\geq |A_N|= N^r\gg_n |E_N|^r.
	\]
	
	Finally, let $G=\Sp(V,\omega)$ with $n=2r$. Let $v_1=e_1,\dots,v_r=e_r,v_{r+1}=f_1,\dots,v_{n}=f_r$ be a basis of $V$ such that $\omega(e_i,e_j)=\omega(f_i,f_j)=0$, $\omega(e_i,f_j)=\delta_{ij}$. Take $E_N$ as in \eqref{eq:E_N_definition} in Example \ref{Exa:sharp_GL_n}.
	Then $E_N$ spans $V$ and $|E_N|=nN$. The subgroup
	\[
	A_N=\{e_i\mapsto \zeta_i e_i,\ f_i\mapsto \zeta_i^{-1}f_i:\zeta_i\in \mu_N\}\subset\Sp(V,\omega)
	\]
	has order $N^r$ and stabilizes $E_N$.  Hence
	$|R_{E_N}^{\Sp(V,\omega)}|\geq |A_N|=N^r=n^{-r}|E_N|^r\gg_n |E_N|^r$.
\end{example}

\section{The special linear group}
\label{Sec:SL_n}

\subsection{Setup for the proposition with multiplicities}

The proof of Theorem \ref{Thm:SL_n_statement_frames} is by induction on $n$. However,  to enable the inductive step, we formulate and prove a stronger statement (Proposition \ref{prop:weighted-general}) with additional data.

Let $V$ be a $d$-dimensional vector space over $k$, and let
$\Delta:V^d\to k$
be a fixed alternating $d$-linear form.  A finite weighted set in $V$ 
is a subset $A\subset V$ together with a function $\mu:A\to\mathbb{N}$. 
We write $W:=\sum_{a\in A}\mu(a)$ for the total weight of $A$, and
$M:=\max_{a\in A}\mu(a)$ for the maximal weight of an element of $A$. If $A=\emptyset$, set $M=0$. 

If $W=0$, all estimates below are trivial, so we usually assume $W>0$.  
Then $1\leq M\leq W$.

Fix an integer $r$ with $0\leq r\leq d-1$.  Put $q:=d-r$.
Let $p_1,\dots,p_r\in V$ be linearly independent fixed vectors (when $r=0$, there are no fixed $p_i$).  Let $u\in V$ be a fixed nonzero vector, and let $\tau\in k^*$.

Define
\begin{equation}\label{eq:def_N_d_r}
	N_{d,r}(A; p_1,\dots,p_r;u;\tau)
	:=
	\sum_{\substack{x_1,\dots,x_q\in A\\
			\Delta(p_1,\dots,p_r,x_1,\dots,x_q)=\tau\\
			u\in \spn(x_1,\dots,x_q)}}
	\mu(x_1)\cdots \mu(x_q);
\end{equation}
this is the weighted number of $q$-tuples $(x_1,\dots,x_q)\in A^q$ satisfying
$\Delta(p_1,\dots,p_r,x_1,
	\dots,x_q)=\tau$ 
and
$u\in \spn(x_1,\dots,x_q)$. 

The central weighted proposition is the following.

\begin{proposition}\label{prop:weighted-general}
	For each pair $(d,r)$ with $d\geq 1$ and $0\leq r\leq d-1$, there exists a constant $C_{d,r}>0$, depending only on $d$ and $r$, such that
	\begin{equation}
		N_{d,r}(A; p_1,\dots,p_r;u;\tau)
		\leq
		C_{d,r} W^{q-1/d}M^{1/d},
		\qquad\text{where $q=d-r$},
		\label{eq:bound_prop_weights}
	\end{equation}
	for every field $k$, every $d$-dimensional $k$-vector space $V$, every  alternating $d$-linear form $\Delta$, every finite weighted set $A\subset V$, every independent list $p_1,
	\dots,p_r$, every nonzero $u\in V$, and every $\tau\in k^*$.
\end{proposition}

The proof of Proposition \ref{prop:weighted-general} will be given in Section \ref{Sec:pr_weighted_prop_technical}. Theorem \ref{Thm:SL_n_statement_frames} follows immediately from Proposition \ref{prop:weighted-general}: 

\begin{proof}[Proof of Theorem \ref{Thm:SL_n_statement_frames}]
	Take $d=n$ and $V=k^n$. Let $\Delta$ be the usual determinant form, and let $A=E$ with all weights equal to $1$. Then $W=|E|$ and $M=1$. Take $r=0$ (so $q=n$), and let $u\in V$ be any nonzero vector. Notice that $\det(x_1,\dots,x_n)=\tau$ implies that $x_1,\dots,x_n$ is a basis of $V$, hence the inclusion $u\in\spn(x_1,\dots,x_n)$ is automatic, once the determinant condition is satisfied. So the count in Theorem \ref{Thm:SL_n_statement_frames} is exactly $N_{n,0}(E;u;\tau)$. Proposition \ref{prop:weighted-general} gives
	\[N_{n,0}(E;u;\tau)\leq C_{n,0}|E|^{n-1/n}.\qedhere\]
\end{proof}

\subsection{Linear algebra preparations}
\label{Sec:preparations_for_weighted_proposition}

\begin{lemma}\label{lem:span-test}
Let $0\leq r\leq d-1$.	Assume
	$
	\Delta(p_1,
	\dots,p_r,x_1,
	\dots,x_q)\neq 0.
	$
Then a vector $u\in V$ lies in $\spn(x_1,
	\dots,x_q)$ if and only if, for every $1\leq i\leq r$,
	\begin{equation}\label{eq:span-test-equations}
		\Delta(p_1,
		\dots,p_{i-1},u,p_{i+1},
		\dots,p_r,x_1,
		\dots,x_q)=0.
	\end{equation}
\end{lemma}

\begin{proof}
The nonvanishing of $\Delta$ immediately implies that
	$p_1,\dots,p_r,x_1,\dots,x_q$ is a basis of $V$.  Write $u$ in this basis:
	\[
	u=\alpha_1p_1+
	\cdots+\alpha_rp_r+\beta_1x_1+\cdots+\beta_qx_q.
	\]
	For fixed $i$, substitute this expression into the left-hand side of \eqref{eq:span-test-equations}. 
All the $p_j$ with $j\neq i$, as well as all $x_j$ are already present among the arguments in $\Delta$.  Thus
	\[
	\Delta(p_1,
	\dots,p_{i-1},u,p_{i+1},
	\dots,p_r,x_1,
	\dots,x_q)
	=
	\alpha_i\Delta(p_1,
	\dots,p_r,x_1,
	\dots,x_q).
	\]
	Since the determinant on the right is nonzero, equation \eqref{eq:span-test-equations} is equivalent to $\alpha_i=0$.  Therefore all equations \eqref{eq:span-test-equations} hold if and only if
	$
	\alpha_1=\cdots=\alpha_r=0$, or, equivalently,
	$u\in \spn(x_1,\dots,x_q)$.
\end{proof}

\begin{lemma}\label{lem:line-identity-general}
	Let $0\leq r\leq d-2$, and let $q=d-r$.  Suppose $p_1,
	\dots,p_r\in V$ are linearly independent, $u\in V$ is nonzero, and $a,b,z_2,
	\dots,z_q\in V$ satisfy
	\begin{equation}\label{eq:line-id-ab}
		\Delta(p_1,\dots,p_r,a,z_2,\dots,z_q)=\tau\neq 0,
		\qquad
		\Delta(p_1,\dots,p_r,b,z_2,\dots,z_q)=0,
	\end{equation}
	and, for each $1\leq i\leq r$,
	\begin{equation}\label{eq:line-id-u-ab}
		\Delta(p_1,\dots,p_{i-1},u,p_{i+1},\dots,p_r,a,z_2,\dots,z_q)=0,
		\qquad
		\Delta(p_1,\dots,p_{i-1},u,p_{i+1},\dots,p_r,b,z_2,\dots,z_q)=0.
	\end{equation}
	Then
	$u\in \spn(z_2,\dots,z_q)$ or $b\in \spn(z_2,\dots,z_q)$.
\end{lemma}

\begin{proof}
	Let
	$S:=\spn(z_2,\dots,z_q)$.
	By the first equation in \eqref{eq:line-id-ab}, the list
	$p_1,\dots,p_r,a,z_2,\dots,z_q$
	is a basis of $V$.  Write
	\[
	b=b_P+\beta a+s_b,
	\qquad
	u=u_P+\alpha a+s_u,
	\]
	where $b_P,u_P\in \spn(p_1,
	\dots,p_r)$ and $s_b,s_u\in S$.  The second equation in \eqref{eq:line-id-ab} forces $\beta=0$, because only the $a$-component of $b$ contributes to $\Delta(p_1,\dots,p_r,b,z_2,\dots,z_q)$.
	Similarly, the first equation in \eqref{eq:line-id-u-ab}, for each $i$, forces all coordinates of $u_P$ to be zero.  
	
	Write
	$b_P=\gamma_1p_1+\cdots+\gamma_rp_r$. 
	Now substitute $u=\alpha a+s_u$ and $b=b_P+s_b$ into the second equation in \eqref{eq:line-id-u-ab}.  For a fixed $i$, all terms vanish except the term in which $u$ contributes $\alpha a$ and $b_P$ contributes $\gamma_i p_i$.  Therefore
	\[
	0=
	\Delta(p_1,
	\dots,p_{i-1},u,p_{i+1},
	\dots,p_r,b,z_2,
	\dots,z_q)
	=
	\pm \alpha\gamma_i\tau.
	\]
	Since $\tau\neq 0$, we have
	$\alpha\gamma_i=0$ for every $i$.
	If $\alpha=0$, then $u=s_u\in S$.  If $\alpha\neq 0$, then all $\gamma_i=0$, so $b=s_b\in S$.  
\end{proof}

\subsection{Proof of the proposition with multiplicities}
\label{Sec:pr_weighted_prop_technical}

We prove Proposition \ref{prop:weighted-general} by induction on the ambient dimension $d$. For a fixed $d$, we use descending induction on $r$. For fixed $d$ and $r$, we perform strong induction on $W$. We can assume that $\Delta$ is not identically $0$. 

Suppose $d=1$. Then necessarily $r=0$ and $q=1$. Set $C_{1,0}=1$. The determinant condition is a nonzero linear equation in one variable, so it selects at most one point of the weighted set.  Thus
$N_{1,0}(A;u;\tau)\leq M=C_{1,0}W^{1-1}M.$

Now fix $d\geq 2$, and assume that Proposition \ref{prop:weighted-general} holds in all dimensions $<d$. We prove the estimates in dimension $d$ by descending induction on $r$.

For the base case $r=d-1$, we set
$C_{d,d-1}=1$. Here $q=1$.  The conditions are
$\Delta(p_1,\dots,p_{d-1},x)=\tau$ and $u\in \spn(x)$.
Since $u\neq 0$, the span condition forces $x=\lambda u$ for some scalar $\lambda$, and the determinant equation determines at most one value of $\lambda$.  Hence
\[N_{d,d-1}(A;p_1,\dots,p_{d-1};u;\tau)\leq M\leq W^{1-1/d}M^{1/d}\]
(using $M\leq W$ on the last step).

Now suppose $0\leq r\leq d-2$, and suppose that the constant $C_{d,r+1}$ and the estimate for the pair $(d,r+1)$ have already been fixed and proved.  We first define the constant for the pair $(d,r)$.  Set $q=d-r$, and define
\[
K_{d,r}:=1+C_{d-1,r},\qquad
B'_{d,r}:=2C_{d,r+1},\qquad
B_{d,r}:=q(1+B'_{d,r}),\qquad
C_{d,r}:=\max\{K_{d,r},B_{d,r}\}.
\]

We prove by strong induction on the integer $W$ that \eqref{eq:bound_prop_weights} holds for every choice of the data in Proposition \ref{prop:weighted-general} with fixed pair $(d,r)$ and total weight $W$. The case $W=0$ is trivial.  Assume $W>0$, and assume that \eqref{eq:bound_prop_weights} holds for every such choice of data with total weight $W'<W$.  Now fix a choice of the data in Proposition \ref{prop:weighted-general}, still with the fixed pair $(d,r)$, such that $A$ has total weight $W$, and let $M$ be the maximum point weight of $A$.  

Define the weight of an affine line $L\subset V$ as
\[
\mu(L):=\sum_{a\in A\cap L}\mu(a).
\]

We will split into cases according to whether there exists an affine line of weight greater than
\[
H:=W^{1/d}M^{(d-1)/d}.
\]
Call such a line rich. If no affine line is rich, then fixing $x_1=a$ gives a direct estimate when $a\in\spn(u)$, while for $a\notin\spn(u)$ we pass to the quotient $V/\spn(a)$ and apply the inductive hypothesis in dimension $d-1$, using that each point in the quotient will have induced weight at most $H$. If a rich affine line exists, we remove it from the weighted set, apply induction to the complement, and estimate separately the tuples meeting the removed line. The latter mechanism is reminiscent of pruning arguments in incidence geometry (see, for example \cite{pinned_distances}). 

\paragraph{Case 1: no rich affine line.}
Assume that
$\mu(L)\leq H$ for every affine line $L\subset V$.  We count tuples by fixing the first variable $x_1=a$ and distinguishing between  $a\in\spn(u)$ and $a\notin\spn(u)$. 

First, the contribution to \eqref{eq:def_N_d_r} from tuples with $a\in\spn(u)$ is 

\begin{align*}
	\sum_{a\in A\cap ku}\mu(a)
	\sum_{\substack{x_2,\dots,x_q\in A\\ \text{conditions hold}}}
	\mu(x_2)\cdots\mu(x_q) 
	&\leq \sum_{a\in A\cap ku}\mu(a) \sum_{x_2,\dots,x_q\in A}\mu(x_2)\dots\mu(x_q)\\
	&= \sum_{a\in A\cap ku}\mu(a) W^{q-1}\\
	&\leq HW^{q-1} &&\text{because }\mu(ku)\leq H\\
	&=W^{q-1+1/d}M^{(d-1)/d} &&\text{by the definition of }H\\
	&\leq W^{q-1/d}M^{1/d} &&\text{because }M\leq W.
\end{align*}

Next, we count the contribution from tuples with $a\notin \spn(u)$. If $a\in \spn(p_1,\dots,p_r)$, then the determinant condition is impossible with $x_1=a$, so there is no contribution.  We may therefore assume
$a\notin \spn(u)$
and $a\notin \spn(p_1,\dots,p_r)$.

Let
\[
\pi_a:V\to \overline V:=V/\spn(a)
\]
be the quotient map.  The images $\overline{p_1},\dots,\overline{p_r}$ of
$p_1,\dots,p_r$ are linearly independent, and $\overline{u}:=\pi_a(u)\neq 0$.  The form $\Delta$ induces an alternating $(d-1)$-linear form $\overline\Delta$ on $\overline V$ by inserting $a$ into $\Delta$ on the $(r+1)$-st slot.

Give each point $\overline y\in \overline{A}:=\pi_a(A)$ the quotient weight
\[
\overline\mu(\overline y):=\sum_{x\in A:\ \pi_a(x)=\overline y}\mu(x).
\]
The quotient weighted set $\overline{A}$ has total weight $W$, and every quotient point has weight at most $H$, because 
each fiber of $\pi_a$ is an affine line parallel to $a$.

Now fix $x_1=a$ with $a$ as above. Then the inner sum in \eqref{eq:def_N_d_r} satisfies
\begin{align*}
	\sum_{\substack{x_2,\dots,x_q\in A\\ \Delta(p_1,\dots,p_r,a,x_2,\dots,x_q)=\tau\\ u\in\spn(a,x_2,\dots,x_q)}}	\mu(x_2)\dots\mu(x_q)
	&=\sum_{\substack{\overline{x_2},\dots,\overline{x_q}\in \overline{A}\\ \overline{\Delta}(\overline{p_1},\dots,\overline{p_r},\overline{x_2},\dots,\overline{x_q})=\tau\\ \overline{u}\in\spn(\overline{x_2},\dots,\overline{x_q})}}	
	\overline{\mu}(\overline{x_2})\dots\overline{\mu}(\overline{x_q})\\
&=N_{d-1,r}(\overline{A};\overline{p_1},\dots,\overline{p_r};\overline{u};\tau) &&\text{by the definition \eqref{eq:def_N_d_r}}\\
	&\leq C_{d-1,r} W^{(q-1)-1/(d-1)}H^{1/(d-1)} &&\text{by the inductive hypothesis}\\
	&=C_{d-1,r}W^{q-1-1/d}M^{1/d} &&\text{by the definition of $H$}. 
\end{align*}

With this bound for each inner sum, estimate $\sum_{a\in A}\mu(a)\leq W$ to obtain the bound
$C_{d-1,r}W^{q-1/d}M^{1/d}$
for the terms in \eqref{eq:def_N_d_r} with $x_1=a\notin\spn(u)$. 

Combining the estimates above for the two types of terms,
\[N_{d,r}(A;p_1,\dots,p_r;u;\tau)\leq 
(1+C_{d-1,r})W^{q-1/d}M^{1/d}
=K_{d,r}W^{q-1/d}M^{1/d}
\leq C_{d,r}W^{q-1/d}M^{1/d}.
\]

\paragraph{Case 2: a rich affine line exists.}
Suppose that an affine line
$L=a+\spn(b)$,
where $b\neq 0$, has weight
\[
h:=\mu(L)>H.
\]
Let $A'$ be obtained from $A$ by deleting all points on $L$.  Its total weight is $W-h<W$, and its maximum point weight is at most $M$.  By the strong induction hypothesis,
\begin{equation}\label{eq:general-induction-Aprime}
	N_{d,r}(A';p_1,\dots,p_r;u;\tau)
	\leq
	C_{d,r}(W-h)^{q-1/d}M^{1/d}.
\end{equation}

It remains to count tuples touching $L$. By symmetry among the $q$ variables, it is enough to count tuples with $x_1\in L$ and then multiply by $q$.   

\noindent\textbf{Claim.} The total contribution to \eqref{eq:def_N_d_r} of tuples with $x_1\in L$ is at most
\[
(1+B'_{d,r})hW^{q-1-1/d}M^{1/d}.
\]

\begin{proof}[Proof of Claim]
	Fix a tuple $Z=(x_2,\dots,x_q)\in A^{q-1}$. 
	For $x_1=a+\lambda b$, the determinant equation
	\begin{equation}
		\label{eq:det_equation_affine_linear}	
		\Delta(p_1,\dots,p_r,x_1,Z)=\tau
	\end{equation}
	is affine-linear in $\lambda$. Whenever the determinant equation holds, Lemma \ref{lem:span-test} says that the span condition
	$u\in\spn(x_1,\dots,x_q)$
	is equivalent to the $r$ affine-linear equations
	\begin{equation}
		\Delta(p_1,\dots,p_{i-1},u,p_{i+1},\dots,p_r,x_1,Z)=0,
		\qquad 1\leq i\leq r.
		\label{eq:span_equation_affine_linear}
	\end{equation}
	There are two types of tuples $Z$ according to whether 
	
	\begin{itemize}
		\item[I)] The determinant equation \eqref{eq:det_equation_affine_linear} and the span equations \eqref{eq:span_equation_affine_linear} are not all identities in $\lambda$;
		\item[II)] The determinant equation \eqref{eq:det_equation_affine_linear} and the span equations \eqref{eq:span_equation_affine_linear} are all identities in $\lambda$.
	\end{itemize}

	\textbf{Contribution from tuples $Z$ of type I.}
	If $Z$ is of type I, then there exists at most one value of $\lambda$ for which \eqref{eq:det_equation_affine_linear} and \eqref{eq:span_equation_affine_linear} hold; thus, there exists at most one $x_1\in A\cap L$ such that $(x_1,Z)$ satisfies the conditions in the summation \eqref{eq:def_N_d_r}. Note that $\mu(x_1)\leq M$. So the contribution to \eqref{eq:def_N_d_r} from such tuples is at most
	\begin{align*}
		\sum_{Z\text{ of type I}}
		\sum_{\substack{x_1\in A\cap L\\ \text{conditions}}}
		\mu(x_1)\cdots\mu(x_q)
		&\leq
		M\sum_{Z\text{ of type I}}\mu(x_2)\cdots\mu(x_q)\\
		&\leq
		MW^{q-1}
		&&\text{summing over all $Z\in A^{q-1}$}\\
		&\leq
		hW^{q-1-1/d}M^{1/d}
		&&\text{using $h>H$}.
	\end{align*}
	
	\textbf{Contribution from tuples $Z$ of type II.} Let $Z$ be a type-II tuple. 
	The determinant equation \eqref{eq:det_equation_affine_linear} is the identity in $\lambda$ precisely when
	\[
	\Delta(p_1,\dots,p_r,a,Z)=\tau,
	\qquad
	\Delta(p_1,\dots,p_r,b,Z)=0.
	\]
	Similarly, the span equation indexed by $i$ is the identity in $\lambda$ precisely when both
	\[
	\Delta(p_1,\dots,p_{i-1},u,p_{i+1},\dots,p_r,a,Z)=0
	\quad\text{and}\quad
	\Delta(p_1,\dots,p_{i-1},u,p_{i+1},\dots,p_r,b,Z)=0
	\]
	hold. Therefore Lemma \ref{lem:line-identity-general} applies and gives
	\begin{equation}
		u\in\spn(Z)
		\qquad\text{or}\qquad
		b\in\spn(Z).
		\label{eq:two_alternatives_for_span_Z}
	\end{equation}
	The list $p_1,\dots,p_r,a$ is independent because
	$\Delta(p_1,\dots,p_r,a,Z)=\tau\neq 0$. 
	
	We conclude that every tuple $Z\in A^{q-1}$ of type II satisfies 
	$\Delta(p_1,\dots,p_r,a,Z)=\tau$ and \eqref{eq:two_alternatives_for_span_Z}. Below we bound the weighted number of type-II tuples $Z$ by two applications of the $(d,r+1)$ inductive hypothesis, using the list $p_1,\dots,p_r,a$ of linearly independent vectors, and either $u$ or $b$ as a distinguished vector contained in $\spn(Z)$.  Namely, we bound the contribution from type-II tuples as follows:
	\begin{align*}
		\sum_{Z\text{ of type II}}
		\sum_{\substack{x_1\in A\cap L\\ \text{conditions}}}
		\mu(x_1)\cdots\mu(x_q)
		&\leq
		h\sum_{Z\text{ of type II}}\mu(x_2)\cdots\mu(x_q)
		&&\text{using $\sum_{x_1\in A\cap L}\mu(x_1)= h$}\\
		&\leq
		\quad h\sum_{\substack{Z\in A^{q-1}\\
				\Delta(p_1,\dots,p_r,a,Z)=\tau\\
				u\in\spn(Z)}}
		\mu(x_2)\cdots\mu(x_q)\\
		&\quad +
		h\sum_{\substack{Z\in A^{q-1}\\
				\Delta(p_1,\dots,p_r,a,Z)=\tau\\
				b\in\spn(Z)}}
		\mu(x_2)\cdots\mu(x_q)
		&&\text{using \eqref{eq:two_alternatives_for_span_Z}}\\
	&=\quad hN_{d,r+1}(A;p_1,\dots,p_r,a;u;\tau)\\ 
	&\quad + hN_{d,r+1}(A;p_1,\dots,p_r,a;b;\tau)
	 &&\text{by the definition \eqref{eq:def_N_d_r}}\\	
		&\leq
		h\cdot 2C_{d,r+1}W^{q-1-1/d}M^{1/d}
		&&\text{by two $(d,r+1)$ estimates}\\
		&=
		B'_{d,r}hW^{q-1-1/d}M^{1/d}
		&&\text{by the definition of $B'_{d,r}$}.
	\end{align*}
	
	Combining this with the contribution from type-I tuples establishes the Claim.
\end{proof}

Multiplying by $q$, all tuples touching $L$ contribute at most
\begin{equation}\label{eq:touching-general}
	B_{d,r}hW^{q-1-1/d}M^{1/d}\leq C_{d,r}hW^{q-1-1/d}M^{1/d}.
\end{equation}

Combining \eqref{eq:general-induction-Aprime} and \eqref{eq:touching-general} gives
\[
N_{d,r}(A;p_1,\dots,p_r;u;\tau)
\leq
C_{d,r}\left( (W-h)^{q-1/d}
+
hW^{q-1-1/d} \right) M^{1/d}\leq C_{d,r} W^{q-1/d}M^{1/d},
\]
where the last step follows from
\[(W-h)^\alpha+hW^{\alpha-1}\leq W^\alpha\]
with $\alpha:=q-1/d$ (divide both sides by $W^\alpha$ and use that $\alpha\geq 2-1/2>1$ and $h/W\in [0,1]$).

\subsection{A short proof for the size of $R_E^{\SL_n}$}
\label{Sec:short_proof_R_E_special_linear}

Let $E\subset V$ be a finite subset that spans $V$. 
We sketch a short proof of the weaker statement $|R_E^{\SL_n}|\ll_n |E|^{n-1/n}$, due to the elegance of the argument. We can assume $n\geq 2$. 

We first prove that the set of values
\[\Delta(E):=\{\det(x_1,\dots,x_n)\ |\ (x_1,\dots,x_n)\in E^n\}\] 
has size $\gg |E|^{1/n}$. To this end, take a basis $u_1,\dots,u_n$ of $V$ with $u_i\in E$. Let $\delta:=\det(u_1,\dots,u_n)$. For a point $x\in E$, write $x=\sum_{i=1}^n a_i(x)u_i$. For each $i$, note that
\[\det(u_1,\dots,u_{i-1},x,u_{i+1},\dots,u_n)=a_i(x)\delta.\]
The left-hand side belongs to $\Delta(E)$, and $\delta\ne 0$, so $a_i(x)$ takes at most $|\Delta(E)|$ values for every $i$.
But $x$ is determined by its coordinates, hence $|E|\leq|\Delta(E)|^n$, 
giving the estimate for $|\Delta(E)|$.

Split $E^n$ as a union of the $|\Delta(E)|$ level sets of the determinant function. Take a $\tau\in \Delta(E)\setminus\{0\}$ such that the level set 
\[D_\tau:=\{(x_1,\dots,x_n)\in E^n\ |\ \det(x_1,\dots,x_n)=\tau\}\]
corresponding to $\tau$ is smallest among all level sets of nonzero determinant values. Then 
\[|E^n|
\geq \sum_{t\in\Delta(E)\setminus\{0\}} |D_t|
\geq 
|\Delta(E)\setminus\{0\}|\ |D_\tau|\geq\frac{|\Delta(E)|}{2}\ |D_\tau|\gg |E|^{1/n}\ |D_\tau|,\]
giving $|D_\tau|\ll |E|^{n-1/n}$. The second-to-last inequality holds since $|\Delta(E)|\geq 2$ (the determinant assumes the value $0$, as well as at least one nonzero value, since $E$ spans $V$). 
The fact that $R_E^{\SL_n}$ acts freely on $D_\tau$ implies $|R_E^{\SL_n}|\leq |D_\tau|\ll |E|^{n-1/n}$.  

\begin{remark}
	\label{Rem:comparison_R_E_versus_T_E_level_sets}
We compare this with the stronger statement Theorem \ref{Thm:SL_n_statement_frames}. To prove that $|R_E^{\SL_n}|\ll_n |E|^{n-1/n}$, it was sufficient to prove that there {\it exists} a level set $D_\tau$ with $\tau\neq 0$ of order $\ll_n |E|^{n-1/n}$, while Theorem \ref{Thm:SL_n_statement_frames} asserts that {\it every} level set corresponding to a nonzero value has size $\ll_n |E|^{n-1/n}$.  
\end{remark}

The exponent $1/n$ in the bound $|\Delta(E)|\gg |E|^{1/n}$ is sharp, since in the end the exponent $n-1/n$ in the bound for $|R_E^{\SL_n}|$ is sharp.

\section{Orthogonal groups in dimensions $2$ and $3$}
\label{Sec:orthogonal_n_2_3}

The conclusion of Theorem \ref{Thm_orthogonal_cases_n_2_3} in the case of $\SO(V,B)$ follows from that for $\OO(V,B)$, so we consider only $G=\OO(V,B)$.

\subsection{Case $n=2$}
\begin{lemma}
	\label{Lem_OO_n_2}
Let $V$ be a vector space of dimension $2$ over a field $k$ with $\operatorname{char}(k)\neq 2$, and let $B$ be a non-degenerate symmetric bilinear form on $V$.  Let $E\subset V$ be a finite subset. Let $T\in M_{2\times 2}(k)$ be a nonsingular symmetric matrix. Then 		\[\#\{	(m_1,m_2)\in E^2	\ |\ \Gr_{B}(m_1,m_2)=T 	\}\leq 2|E|.	\]
\end{lemma}

\begin{proof}
Let $T=(t_{ij})_{1\leq i,j\leq 2}$. 
Fix $m_1\in E$ with $B(m_1,m_1)=t_{11}$ (there are at most $|E|$ possibilities for $m_1$). We will prove that there exist at most $2$ points $m_2\in E$ such that 
$B(m_1,m_2)=t_{12}$ and $B(m_2,m_2)=t_{22}$.

If $m_1=0$, then a point $m_2$ does not exist (otherwise $t_{11}=t_{12}=0$ and $T$ would be singular). So assume $m_1\neq 0$. The map $V\to k$, $y\mapsto B(m_1,y)$ is surjective, with kernel the one-dimensional space $m_1^\perp$. Let $v$ be a basis of $m_1^\perp$. We now describe the solution set of
\begin{equation}
	B(m_1,y)=t_{12}\quad\text{and}\quad B(y,y)=t_{22}. 	
	\label{eq:orthogonal_case_n_2_conditions_m_1_m_2}	
\end{equation}	

Let $y_0$ be a solution to $B(m_1,y_0)=t_{12}$. Then the set of solutions to $B(m_1,y)=t_{12}$ is $y_0+\spn(v)$. For $y=y_0+\lambda v$, the second equation in (\ref{eq:orthogonal_case_n_2_conditions_m_1_m_2})
becomes $B(y_0+\lambda v,y_0+\lambda v)=t_{22}$, or, equivalently, 
\[B(v,v)\lambda^2+2B(y_0,v)\lambda+B(y_0,y_0)-t_{22}=0.\]
If this polynomial in $\lambda$ is not identically $0$, then it has at most $2$ roots $\lambda\in k$. Suppose the polynomial is identically $0$. In particular, $B(v,v)=0$, $B(y_0,v)=0$, and $B(y_0,y_0)=t_{22}$. Then the Gram matrix $\Gr_B(m_1,y_0)=T$ is nonsingular; in particular, $m_1,y_0$ is a basis of $V$. But the nonzero vector $v$ is orthogonal to both $m_1$ and $y_0$, hence $v\in V^\perp$, contradicting the non-degeneracy of $B$. 
\end{proof}

\subsection{Case $n=3$ with a nonzero principal minor}
\label{Subsec:Orth_n_3}

Throughout this and the next subsection, $V$ is a $3$-dimensional vector space over a
field $k$ with $\operatorname{char}(k)\neq 2$, and $B$ is a non-degenerate symmetric bilinear form on $V$.

\begin{proposition}
	\label{Prop:orthogonal_n_3_nonzero_principal_minor}
	Let $V$ be a $3$-dimensional vector space over $k$, and let $B$ be a non-degenerate symmetric bilinear form on $V$. Let $E\subset V$ be a finite subset. Let $S\in M_{3\times 3}(k)$ be a nonsingular
	symmetric matrix. Then
	\[
	\#\{(x_1,x_2,x_3)\in E^3\ |\ \Gr_B(x_1,x_2,x_3)=S\}
	\ll  |E|^{3/2},
	\]
with an absolute implied constant. 	
\end{proposition}

We now prove Proposition \ref{Prop:orthogonal_n_3_nonzero_principal_minor} in the case when some principal $2\times 2$ minor of $S$ is nonzero. The case when all principal $2\times 2$ minors of $S$ vanish will be treated in Section \ref{Subsec:all_principal_minors_vanish}.

The argument will be based on the following pair-counting estimate.

\begin{lemma}
	\label{lem:binary-Gram}
	Let $a,b,c\in k$ with $ac-b^2\neq 0$.
Then for every finite $E\subset V$, we have
	\[
	\#\{(x,y)\in E^2:B(x,x)=a,\ B(x,y)=b,\ B(y,y)=c\}
	\ll |E|^{3/2},
	\]
	with an absolute implied constant. 
\end{lemma}

The proof of this lemma is inspired by the argument in \cite[Section~5.2]{Do}. Namely, to study the unit-minor problem, Do constructs a $d$-partite hypergraph in which a hyperedge corresponds to a $d$-tuple $(x_1,\dots,x_d)$ of points from a given set with $\det(x_1,\dots,x_d)=1$. She then proves that the hypergraph is $K_{2,\dots,2}$-free and applies a Zarankiewicz-type bound. For our purposes, we need only the classical K\H{o}v\'ari--S\'os--Tur\'an theorem: see \cite[Theorem~1.4.2]{Yufei_Zhao} for a modern statement and proof, or \cite{KST} for the original source. 

\begin{theorem}[K\H{o}v\'ari--S\'os--Tur\'an theorem]
	\label{Thm:KST}	
	Let $s,t\geq 2$. Let $G$ be a graph on at most $N$ vertices.
	Suppose that $G$ contains no $K_{s,t}$-subgraph. Then the number of edges in $G$ is $\ll_{s,t} N^{2-1/s}$. 
\end{theorem}

\begin{proof}[Proof of Lemma \ref{lem:binary-Gram}] 
	 We will apply Theorem \ref{Thm:KST} with $s=2$, $t=3$. 

Build a bipartite graph $G$ whose left vertex set is
$E_a=\{x\in E\ |\ B(x,x)=a\}$ and  whose right vertex set is
$E_c=\{y\in E\ |\ B(y,y)=c\}$.	Join $x\in E_a$ to $y\in E_c$ if
$B(x,y)=b$.	The number of edges in $G$ is exactly the number of ordered pairs $(x,y)\in E^2$ in the lemma.
	
{\bf Case 1: $b\neq0$.} We will prove that $G$ is $K_{2,3}$-free.  Then Theorem~\ref{Thm:KST}, with
	$(s,t)=(2,3)$, gives the bound $\ll |E|^{3/2}$ for the number of edges. 
	
Fix two distinct left vertices $m$ and $m_1$. We will bound the number of their common neighbors: points $y\in E_c$ that satisfy the two affine linear equations
\begin{equation}
B(m,y)=b\quad\text{and}\quad B(m_1,y)=b.	
\label{eq:orthogonal_case_n=3_case_b_neq_0_the_2_lin_eqns}
\end{equation}

If $m_1=\lambda m$ for some $\lambda\in k$ (necessarily $\lambda\neq 1$), then a common neighbor $y\in E_c$ would satisfy $b=B(m_1,y)=\lambda B(m,y)=\lambda b$, contradicting $b\neq 0$. Thus $m$ and $m_1$ have no common neighbors in this case. 

So assume that $m$ and $m_1$ are linearly independent. Let $y_0\in V$ be such that $B(m,y_0)=B(m_1,y_0)=b$ (if no such $y_0$ exists, then there are no common neighbors). Then the solution set to \eqref{eq:orthogonal_case_n=3_case_b_neq_0_the_2_lin_eqns} with $y\in V$ is $y_0+\spn(m,m_1)^\perp$. Let $v$ be a basis of $\spn(m,m_1)^\perp$. 

Therefore a point in $E_c$ that satisfies \eqref{eq:orthogonal_case_n=3_case_b_neq_0_the_2_lin_eqns} is described as $y_0+\lambda v$ for some $\lambda\in k$ such that $B(y_0+\lambda v, y_0+\lambda v)=c$, or, equivalently, 
\[B(v,v)\lambda^2+2B(y_0,v)\lambda+B(y_0,y_0)-c=0.\]

If this polynomial in $\lambda$ is not identically $0$, it has at most $2$ roots $\lambda\in k$, and so $m$ and $m_1$ have at most two common neighbors. Otherwise, $B(v,v)=0$, $B(y_0,v)=0$, and $B(y_0,y_0)=c$.

In the latter case, the Gram matrix $\Gr_B(m, y_0)=\begin{pmatrix}
	a & b\\ b & c
\end{pmatrix}$
 is nonsingular, hence $m$ and $y_0$ are linearly independent and the restriction of $B$ to $\spn(m,y_0)$ is non-degenerate. But then so is the restriction of $B$ to $\spn(m,y_0)^\perp$. Since $v\neq 0$ and $v\in\spn(m,y_0)^\perp,$ we have that $v$ is a basis of $\spn(m,y_0)^\perp$. However, $B(v,v)=0$, contradicting the non-degeneracy of $B$ on $\spn(m,y_0)^\perp$.  
 
Therefore $m$ and $m_1$ have at most $2$ common neighbors. Analogously, any two right vertices have at most $2$ common neighbors, so $G$ is $K_{2,3}$-free.

	{\bf Case 2: $b=0$.}	Then $a,c\neq 0$. For $\lambda\in k$ and $m,\lambda m\in E_a$, we have $\lambda^2 a=\lambda^2 B(m,m)=B(\lambda m,\lambda m)=a$, so $\lambda=\pm 1$. 
Notice that if $m,-m\in E_a$, then $m$ and $-m$ have equal degrees in $G$. Let $E_a^+$ be obtained from $E_a$ by removing one element from each pair $\{m,-m\}$ of elements in $E_a$. Define $E_c^+$ similarly. Then $G$ has at most $4$ times as many edges as the induced subgraph $G^+$ of $G$ on the vertex set $E_a^+\cup E_c^+$, and any two vertices in $E_a^+$ or in $E_c^+$ are linearly independent. 

The proof from Case 1 carries over and yields that $G^+$ is $K_{2,3}$-free.  Theorem \ref{Thm:KST} implies that $G^+$ has $\ll |E|^{3/2}$ edges, and the same bound (with a factor of $4$) will also hold for $G$.   
\end{proof}

\begin{proof}[Proof of Proposition \ref{Prop:orthogonal_n_3_nonzero_principal_minor} in the case when some principal $2\times 2$ minor of $S$ is nonzero]
	Let $S=(s_{ij})_{1\leq i,j\leq 3}$.
After relabelling the coordinates, assume that
$
	\det\begin{pmatrix}s_{11}&s_{12}\\ s_{12}&s_{22}\end{pmatrix}\neq 0.
$
	
 By Lemma~\ref{lem:binary-Gram}, the number of ordered pairs
	$(m_1,m_2)\in E^2$ satisfying
	\[
	B(m_1,m_1)=s_{11},\qquad
	B(m_1,m_2)=s_{12},\qquad
	B(m_2,m_2)=s_{22}
	\]
	is $\ll |E|^{3/2}$.	
	For each such pair, we prove that there are at most $2$ points $y\in E$ satisfying
\begin{equation}
	B(m_1,y)=s_{13}, \quad B(m_2,y)=s_{23}, \quad B(y,y)=s_{33}. 
\label{eq:orthogonal_n_3_last_component_system_eqns}	
\end{equation}	 
	
Let $y_0\in V$ be such that $B(m_1,y_0)=s_{13}$, $B(m_2,y_0)=s_{23}$. Then the set of solutions of the first two equations in \eqref{eq:orthogonal_n_3_last_component_system_eqns} is $y_0+\spn(m_1,m_2)^\perp$. Since $\Gr_B(m_1,m_2)=\begin{pmatrix}
	s_{11} & s_{12}\\ s_{21} & s_{22}
\end{pmatrix}$ is nonsingular, $m_1$ and $m_2$ are linearly independent. Then $\spn(m_1,m_2)^\perp$ is one dimensional; let $v$ be a basis. So the solutions of \eqref{eq:orthogonal_n_3_last_component_system_eqns} consist of all $y_0+\lambda v$ with $\lambda\in k$ such that $B(y_0+\lambda v, y_0+\lambda v)=s_{33}$, or, equivalently, 
\begin{equation}
	B(v,v)\lambda^2+2B(y_0,v)\lambda+B(y_0,y_0)-s_{33}=0.
\label{eq:quadratic_lambda_case_orthogonal_n_3_third_component}	
\end{equation} 	
Since $\Gr_B(m_1,m_2)$ is nonsingular, we know that the restriction of $B$ to $\spn(m_1,m_2)$ is non-degenerate. But then so is the restriction of $B$ to $\spn(m_1,m_2)^\perp=\spn(v)$. This implies $B(v,v)\neq 0$. But then \eqref{eq:quadratic_lambda_case_orthogonal_n_3_third_component} is a quadratic equation, hence it has at most $2$ solutions $\lambda\in k$.
\end{proof}

\subsection{Case $n=3$ and each $2\times 2$ principal minor of $S$ is $0$}
\label{Subsec:all_principal_minors_vanish}

We now prove Proposition \ref{Prop:orthogonal_n_3_nonzero_principal_minor} in the case when each principal $2\times 2$ minor of $S$ is $0$. 

Let
\[ S=
\begin{pmatrix}
	a&p&q\\
	p&b&r\\
	q&r&c
\end{pmatrix}.\]	
The condition on the minors implies $ab=p^2$, $ac=q^2$, and $bc=r^2$. Multiply these to obtain $(abc)^2=(pqr)^2$, so $abc=\pm pqr$. But since $\det S=2(pqr-abc)\neq 0$, we must have $pqr=-abc$. Moreover, $a$, $b$, $c$, $p$, $q$, $r$ are all nonzero. 

Define a new bilinear form $B_0:=a^{-1}B$, and let
\[M:=\begin{pmatrix}
	1&1&1\\
	1&1&-1\\
	1&-1&1
\end{pmatrix}.\] 

Let $\lambda_1=1$, $\lambda_2=a/p$, $\lambda_3=a/q$,
and set $A_i=\lambda_i E$. We have a bijection
\[\{(x_1,x_2,x_3)\in E^3:\Gr_B(x_1,x_2,x_3)=S\}\simeq\{(y_1,y_2,y_3)\in A_1\times A_2\times A_3\ |\ \Gr_{B_0}(y_1,y_2,y_3)=M\}\]
given by $y_i=\lambda_i x_i$.

Therefore Proposition \ref{Prop:orthogonal_n_3_nonzero_principal_minor} will follow from the following.

\begin{proposition}
	\label{Prop:bound_Gram_M_bilinear_B}
Let $V$ be a $3$-dimensional vector space over $k$, and let $B$ be a non-degenerate symmetric bilinear form on $V$.	Let $A_1,A_2,A_3$ be finite subsets of $V$. Then
	\begin{equation}
		\#\{(x,y,z)\in A_1\times A_2\times A_3\ |\ \Gr_{B}(x,y,z)=M\}\leq 2|A_1|^{1/2}|A_2|^{1/2}|A_3|^{1/2}.
		\label{eq:count_Gram_M_bilinear_B}
	\end{equation}
\end{proposition}

We begin with some reductions and preparations. 

We can assume that there exists a triple $(x_1,x_2,x_3)\in A_1\times A_2\times A_3$ with $\Gr_B(x_1,x_2,x_3)=M$. 

Consider the space $W=k^3$, equipped with the bilinear form
\[\langle(X,Y,Z),(X',Y',Z')\rangle=XX'+YZ'+ZY'.\]
The corresponding quadratic form is
\[Q(X,Y,Z)=X^2+2YZ.\]
Consider the following basis of $W$:
$w_1=(1,0,0)$, $w_2=(1,1,0)$, $w_3=(1,0,-2)$.
Since $\Gr_{\langle ,\rangle}(w_1,w_2,w_3)=M$,
the unique linear map $T:W\to V$ sending $w_i$ to $x_i$ is an isometry. 
Let $B_i=T^{-1}A_i$.
The set described in the LHS of (\ref{eq:count_Gram_M_bilinear_B}) is in bijection with 
\[\{(y_1,y_2,y_3)\in B_1\times B_2\times B_3\ |\ \Gr_{\langle , \rangle}(y_1,y_2,y_3)=M\}.\]
Therefore, it suffices to prove Proposition \ref{Prop:bound_Gram_M_bilinear_B} for $(W,\langle,\rangle)$. To say that $(x,y,z)\in A_1\times A_2\times A_3$ satisfies	$\Gr_{\langle,\rangle}(x,y,z)=M$ is to say that 
\begin{equation}
	\label{eq:interpret_rulings}
	\begin{gathered}
		Q(x)=Q(y)=Q(z)=1,\\
		\langle x,y\rangle=1,\qquad
		\langle x,z\rangle=1,\qquad
		\langle y,z\rangle=-1.
	\end{gathered}
\end{equation}

We now parametrize the quadric surface $\{Q=1\}$. The Segre embedding $\P^1\times\P^1\to\P^3_{[A:B:C:D]}$, $([r_0:r_1],[s_0:s_1])\mapsto [r_0s_0:r_0s_1:r_1s_0:r_1s_1]$ identifies in turn
\begin{equation}
	\P^1\times \P^1\xrightarrow{\simeq} \{AD=BC\}\simeq
\{[X:Y:Z:T]\in\P^3\ |\ X^2+2YZ=T^2\},
\label{eq:Segre}
\end{equation}
where we used the change of variables $T=A+D$, $X=A-D$, $Y=B$, $Z=2C$. 
On the affine chart $\{T\neq 0\}$, the RHS in \eqref{eq:Segre} becomes the affine surface $\{Q=1\}$.

For $r=[r_0:r_1]\in \P^1(k)$ and $s=[s_0:s_1]\in\P^1(k)$, define
\[
\Delta(r,s):=r_0s_0+r_1s_1
\]
and
\[
[r,s]:=r_0s_1-r_1s_0.
\]
Note that $[r,s]=0$ if and only if $r=s$. 

Let
\[\Omega:=\{(r,s)\in \P^1(k)\times \P^1(k):\Delta(r,s)\neq 0\}.\]

Since $\Delta(r,s)=A+D=T$, the set $\Omega$ is the preimage under the composite map in \eqref{eq:Segre} of
the affine chart $\{T\neq 0\}$. Therefore \eqref{eq:Segre} induces a bijection
\begin{align*}
	\Phi:	\Omega &\longrightarrow \{v\in W\ |\ Q(v)=1\}\\
	(r,s) &\longmapsto \left(
	\frac{r_0s_0-r_1s_1}{\Delta(r,s)},
	\frac{r_0s_1}{\Delta(r,s)},
	\frac{2r_1s_0}{\Delta(r,s)}
	\right).
\end{align*}
Moreover, a direct computation shows that for $(r,s), (u,t)\in\Omega$, we have
\begin{equation}
	\langle\Phi(r,s),\Phi(u,t)\rangle=1-\frac{2[r,u][s,t]}{\Delta(r,s)\Delta(u,t)}.
	\label{eq:dot_product_Phis}
\end{equation} 
Consequently, 
\begin{equation}
	\langle\Phi(r,s),\Phi(u,t)\rangle=1\quad\text{if and only if}\quad
	r=u\ \text{or}\ s=t.
	\label{eq:when_dot_prod_equals_1}
\end{equation} 	

\begin{proof}[Proof of Proposition \ref{Prop:bound_Gram_M_bilinear_B} for $(V,B)=(W,\langle,\rangle)$] We study \eqref{eq:interpret_rulings} using the parametrization of $\{Q=1\}$.  
Write $x=\Phi(r,s)$, $y=\Phi(r',s')$, $z=\Phi(r'',s'')$, with
	$(r,s), (r',s'), (r'',s'')\in\Omega$. Then (\ref{eq:when_dot_prod_equals_1}) applied to $\langle x,y\rangle=1$,
	$\langle x,z\rangle=1$ implies that 
	\[r=r'\quad\text{or}\quad s=s'\]
	and
	\[r=r''\quad\text{or}\quad s=s''.\]
	Note that $x\neq y$, since $\langle x,z\rangle\neq\langle y,z\rangle$. Similarly, $x\neq z$. Thus, ``or" is in fact an ``xor" in both displayed lines above.  	
	
	Next, note that $y$ and $z$ cannot have equal first or second components because otherwise (\ref{eq:when_dot_prod_equals_1}) would imply $\langle y,z\rangle=1$, contradicting $\langle y,z\rangle=-1$. 
	
	We conclude that either $r=r'$ and $s=s''$, or $s=s'$ and $r=r''$. We treat the former case (the factor of $2$ on the RHS in (\ref{eq:count_Gram_M_bilinear_B}) accounts for allowing each of these two cases).  
	
	So
	\[x=\Phi(r,s), y=\Phi(r,t), z=\Phi(u,s) \quad\text{with}\quad (r,s),(r,t),(u,s)\in\Omega.\]
	
	By (\ref{eq:dot_product_Phis}), the condition $\langle y,z \rangle=-1$ is equivalent to 
	\[[r,u][t,s]=\Delta(r,t)\Delta(u,s).\]
The difference $[r,u][t,s]-\Delta(r,t)\Delta(u,s)$ can be factored as
	\[[r,u][t,s]-\Delta(r,t)\Delta(u,s)=-\Delta(r,s)\Delta(u,t).\]
Note that $\Delta(r,s)\neq 0$, since $(r,s)\in\Omega$. Therefore
	$\langle y,z\rangle=-1$ is equivalent to 
	$\Delta(u,t)=0$.
	
	For every $u=[u_0:u_1]\in\P^1(k)$, note that $\varphi(u):=[-u_1:u_0]$ is the unique $t\in\P^1(k)$ such that $\Delta(u,t)=0$. 
	
	So, we are reduced to counting triples $(r,s,u)\in\P^1(k)\times \P^1(k)\times \P^1(k)$ such that $(r,s), (r,\varphi(u)), (u,s)\in\Omega$ and 
	\[\Phi(r,s)\in A_1,\quad\Phi(r,\varphi(u))\in A_2,\quad \Phi(u,s)\in A_3.\]
	In other words, we are counting triples $(r,s,u)\in \P^1(k)\times \P^1(k)\times \P^1(k)$ which satisfy the following:
	\begin{itemize}
		\item $(r,s)\in B_1:=\Phi^{-1}(A_1)$;
		\item $(r,u)\in B_2:=\{(r,u)\in\P^1(k)\times\P^1(k)\ |\ (r,\varphi(u))\in\Omega\ \text{and}\ \Phi(r,\varphi(u))\in A_2\}$;
		\item $(s,u)\in B_3:=\{(s,u)\in\P^1(k)\times\P^1(k)\ |\ (u,s)\in\Omega\ \text{and}\ \Phi(u,s)\in A_3\}$.  
	\end{itemize}
	Note that each $B_i$ injects in $A_i$, so $|B_i|\leq |A_i|$. Therefore (\ref{eq:count_Gram_M_bilinear_B}) follows from Lemma \ref{Lem:finite_Loomis_Whitney} below.
\end{proof}	

The lemma below is a finite version of the Loomis--Whitney inequality \cite{LoomisWhitney1949} in dimension~$3$ and is a particular case of 
\cite[Theorem~2.1]{Finner1992}. 

\begin{lemma}
	\label{Lem:finite_Loomis_Whitney}	
	Let $R$, $S$, $U$ be arbitrary sets, and let $B_1\subset R\times S$, $B_2\subset R\times U$, and $B_3\subset S\times U$ be finite subsets. Then
	\begin{equation}
		\#\{(r,s,u)\in R\times S\times U\ |\ (r,s)\in B_1, (r,u)\in B_2, (s,u)\in B_3\}\leq |B_1|^{1/2}|B_2|^{1/2}|B_3|^{1/2}.
		\label{eq:Loomis_Whitney_3d}	
	\end{equation}	
\end{lemma}	

\begin{proof}
Only finitely many elements of $R$, $S$, and $U$ occur as coordinates of points in $B_1$, $B_2$, and $B_3$, so we may assume without loss of generality that $R$, $S$, and $U$ are finite. The result follows from 	
\cite[Theorem~2.1]{Finner1992} applied to these finite sets, equipped with counting measure, and to the indicator functions of $B_1$, $B_2$, and $B_3$.  We give a short independent proof for completeness.	
	
Let $T$ denote the set whose cardinality appears on the LHS of (\ref{eq:Loomis_Whitney_3d}). For each $(r,s)\in B_1$, let
	\[N(r,s)=\#\{u\in U\ |\ (r,s,u)\in T\}.\]
	Then $|T|=\sum_{(r,s)\in B_1}N(r,s)$ and the Cauchy--Schwarz inequality implies
	\[|T|^2\leq |B_1|\sum_{(r,s)\in B_1}N(r,s)^2.\] 	
	Thus it suffices to establish the inequality
	\begin{equation}
		\sum_{(r,s)\in B_1}N(r,s)^2\leq |B_2||B_3|.
		\label{eq:Loomis_Whitney_reduces_to_this}	
	\end{equation}
	The LHS of (\ref{eq:Loomis_Whitney_reduces_to_this}) is the cardinality of the set
	\[D:=\{(r,u,s,u')\in R\times U\times S\times U\ |\ (r,s,u)\in T, (r,s,u')\in T\}.\]
	But the map $D\to B_2\times B_3$, $(r,u,s,u')\mapsto ((r,u), (s,u'))$ is an injection. Therefore $|D|\leq |B_2||B_3|$, establishing (\ref{eq:Loomis_Whitney_reduces_to_this}).  
\end{proof}	

\begin{proof}[Proof of Proposition \ref{Prop:orthogonal_n_3_nonzero_principal_minor}]
The case when  some principal $2\times 2$ minor of $S$ is nonzero was proved in Section \ref{Subsec:Orth_n_3}. The case when all principal $2\times 2$ minors vanish was proved above. 
\end{proof}

\begin{proof}[Proof of Theorem \ref{Thm_orthogonal_cases_n_2_3}]
The case $n=2$ follows from Lemma \ref{Lem_OO_n_2}, and the case $n=3$ follows from Proposition \ref{Prop:orthogonal_n_3_nonzero_principal_minor}. 	
\end{proof}

\section{Orthogonal groups via Brascamp--Lieb}
\label{Sec:Brascamp_Lieb}

\subsection{Verifying the Brascamp--Lieb condition for an orthogonal group}

We establish the following

\begin{proposition}
	\label{Prop:OO_satisfies_Brascamp_Lieb}	
	Let $k$ be an algebraically closed field with $\operatorname{char}(k)\neq 2$, and let $V=k^n$. Let $B$ be a non-degenerate symmetric bilinear form on $V$. Let $u_1,\dots,u_n\in V$ form a basis of $V$ as a vector space over $k$. Consider $\pi_j: \OO(V,B)\to \A^n$, $g\mapsto gu_j$ for $j=1,\dots,n$. Let $Z$ be an irreducible subvariety of $\OO(V,B)$. Then 
	\[\dim Z\leq\frac{1}{2}\sum_{j=1}^n \dim \overline{\pi_j(Z)}.\] 
\end{proposition}

We begin by stating the linear algebra result that Proposition \ref{Prop:OO_satisfies_Brascamp_Lieb} will be reduced to.

\begin{lemma}[Lemma 4 in \cite{Pendavingh}]
	\label{Lem:lin_alg_projections_inequality}
	Let $k$ be a field. Let $\Omega$ be a finite set, and let $A_1,\dots,A_m$ be subsets of $\Omega$. For each $i$, consider the linear projection map $\operatorname{pr}_i:k^\Omega\to k^{A_i}$, $(c_x)_{x\in\Omega}\mapsto (c_x)_{x\in A_i}$. Suppose $r$ is such that
	$A_1,\dots,A_m$ form an exact $r$-cover of $\Omega$ --- that is, for each $x\in \Omega$, we have $\#\{i\in\{1,\dots,m\}\ |\ x\in A_i\}=r$. Let $L\subset k^\Omega$ be a vector subspace. Then
	\[r\dim L\leq\sum_{i=1}^m\dim\operatorname{pr}_i(L).\]
\end{lemma}

\begin{proof}[Proof of 
	Proposition \ref{Prop:OO_satisfies_Brascamp_Lieb}.]
	We will study the differentials of the $\pi_i$ using the following ingredients.  
	
	The Lie algebra of $\OO(V,B)$ can be described as 
	\[\mathfrak{o}(V,B)=\{X\in \End(V)\ |\ B(Xu,v)+B(u,Xv)=0\ \text{for all $u,v\in V$}\}.\]
	
	We will use the usual identification $\bigwedge^2 V^*\simeq\Alt^2(V,k)$ induced from the map $(V^*)^2\to \Alt^2(V,k)$, $(\lambda,\mu)\mapsto \lambda\wedge\mu$, where, for $(u,v)\in V^2$, $(\lambda\wedge\mu)(u,v)=\det\begin{pmatrix}
		\lambda(u) & \lambda(v)\\
		\mu(u) & \mu(v)
	\end{pmatrix}$, as well as the standard identification $\mathfrak{o}(V,B)\simeq \bigwedge^2 V^*$.
	
	An element $x\in V$ induces a contraction-by-$x$ map 
	\[C_x: \bigwedge^2 V^*\longrightarrow V^*,\qquad \alpha\longmapsto \alpha(x,\cdot);\]
	explicitly, for $\lambda,\mu\in V^*$, we have
	$C_x(\lambda\wedge\mu)=\lambda(x)\mu-\mu(x)\lambda$. Notice that $\operatorname{im}(C_x)\subset\operatorname{Ann}(x)$, where $\operatorname{Ann}(x):=\{\lambda\in V^* \ |\  \lambda(x)=0\}$.

	Let $\Omega$ be the set of $2$-element subsets of $\{1,\dots,n\}$. For each $i=1,\dots,n$, let $A_i\subset \Omega$ be the set of $2$-element subsets of $\{1,\dots,n\}$ that contain $i$. Then $A_1,\dots,A_n$ form an exact $2$-cover of $\Omega$. 
	
	For any $g\in \OO(V,B)(k)$, the commutative diagram below identifies the differential maps $d\pi_i|_g$ first with the contraction maps $C_{u_i}$ and then in turn with the projection maps $\operatorname{pr}_i$ described in Lemma \ref{Lem:lin_alg_projections_inequality}. 
	
	\[
	\begin{tikzcd}[column sep=huge,row sep=huge]
		T_g\OO(V,B)
		\arrow[rr, "{d\pi_i|_g}"]
		\arrow[d, "\simeq"', "{dL_{g^{-1}}|_g}"]
		&& T_{g u_i}\A(V)
		\arrow[rr, "\simeq", "{\mathrm{can}}"']
		\arrow[d, "\simeq"', "{d(g^{-1})|_{g u_i}}"]
		&& V
		\arrow[d, "\simeq"', "{g^{-1}}"]
		\\
		\mathfrak{o}(V,B)
		\arrow[rr, "{d\pi_i|_e}"]
		\arrow[rrrr, bend right=14, "{X\mapsto Xu_i}"']
		\arrow[d, "\simeq"', "{X\mapsto B(X\cdot,\cdot)}"]
		&& T_{u_i}\A(V)
		\arrow[rr, "\simeq", "{\mathrm{can}}"']
		&& V
		\arrow[d, "\simeq"', "{v\mapsto B(v,\cdot)}"]
		\\
		\bigwedge^2 V^*
		\arrow[rrr, "{C_{u_i}}", "{\omega\mapsto \omega(u_i,\cdot)}"']
		\arrow[d, "\simeq"']
		&&& \operatorname{Ann}(u_i)
		\arrow[r, hookrightarrow]
		\arrow[d,"\simeq"',"{\psi_i}"]
		& V^*
		\\
		k^\Omega
		\arrow[rrr, "{\operatorname{pr}_i}"]
		&&& k^{A_i}
	\end{tikzcd}
	\]
	
	Here the left vertical isomorphism in the bottom rectangle is induced by the basis $(e_i\wedge e_j)_{1\leq i<j\leq n}$ of $\bigwedge^2 V^*$, where $e_1,\dots,e_n$ is the basis of $V^*$ dual to $u_1,\dots,u_n$, and $\psi_i$ is defined as
	\[
	\psi_i(\lambda)_{\{i,j\}}=
	\begin{cases}
		\lambda(u_j),  & i<j,\\
		-\lambda(u_j), & j<i.
	\end{cases}
	\]
	
	Therefore, for any linear subspace $L\subset T_g\OO(V,B)$, Lemma \ref{Lem:lin_alg_projections_inequality} implies 
	\begin{equation}
		\label{eq:main_estimate_exterior_alg_subspace_Lie_alg}	
		2\dim L\leq \sum_{i=1}^n \dim d\pi_i|_g(L).
	\end{equation}
	
	For $i=1,\dots,n$, let $Y_i=\overline{\pi_i(Z)}$. 
	The restriction of $\pi_i$ to $Z$ factors through a
	dominant morphism $Z\to Y_i$. Since the smooth locus of each $Y_i$ is a dense open and $Z$ is irreducible, there exists a $g\in Z$ such that $\pi_i(g)$ is a smooth point of $Y_i$ for each $i=1,\dots,n$. 
	
	Consider $L:=T_g Z\subset T_g\OO(V,B)$. The differential of $\pi_i|_Z$ at $g$ factors through $T_{\pi_i(g)}Y_i$. Therefore
	\begin{equation}
		\label{eq:exterior_alg_tg_space_in_factorization}	
		\dim d\pi_i|_g(L)\leq \dim T_{\pi_i(g)}Y_i=\dim Y_i.
	\end{equation}
	
	Combining
	\eqref{eq:main_estimate_exterior_alg_subspace_Lie_alg} and
	\eqref{eq:exterior_alg_tg_space_in_factorization}, we obtain
	
	\[
	2\dim Z\leq 2\dim L \leq\sum_{i=1}^n \dim d\pi_i|_g(L) \leq
	\sum_{i=1}^n \dim Y_i.\qedhere\]
	
\end{proof}

\subsection{Verifying the Brascamp--Lieb condition in a family}
\label{Sec:BL_family}

\begin{proof}[Proof of Corollary \ref{Cor:General_case_orthogonal}]
We can assume that $k$ is algebraically closed. It suffices to treat the case $G=\OO(V,B)$. The case $n=1$ is trivial, so assume $n\geq 2$. 

If $B_0$ and $B$ are two non-degenerate symmetric bilinear forms on $V$, there exists $L\in\GL_n(k)$ such that
$B(v,w)=B_0(Lv,Lw)$ for all $v,w\in V$. Then $L$ induces a bijection
\[T_{E,u}^{\OO(V,B)}\xrightarrow{\simeq}T_{L(E),L(u)}^{\OO(V,B_0)},\quad g\mapsto LgL^{-1}.\]  
So we fix a non-degenerate symmetric bilinear form $B_0$ on $V$. 

Proposition \ref{Prop:OO_satisfies_Brascamp_Lieb} along with Conjecture \ref{Conj:Brascamp_Lieb} immediately implies a weaker version of Corollary \ref{Cor:General_case_orthogonal} in which the implicit constant can depend on the basis $u_1,\dots,u_n$. In order to eliminate that dependence, we work with a family that represents a universal basis and apply Proposition \ref{Prop:OO_satisfies_Brascamp_Lieb} over the geometric generic fiber.   

For $j=1,\dots,n$, consider the map $\Pi_j: \GL_n\times\OO(V,B_0)\to \GL_n\times\A^n_k$ defined by $(T,g)\mapsto (T,gTe_j)$; here $e_j$ is the $j$-th standard basis vector of $k^n$, so $Te_j$ is the $j$-th column of $T$.

{\bf Claim.} Let $n\geq 2$. Let $Z$ be an integral subvariety of $\GL_n\times \OO(V,B_0)$. Let $Y_j$ denote the scheme-theoretic image of $Z\to \GL_n\times\OO(V,B_0)\xrightarrow{\Pi_j}\GL_n\times\A^n_k$, for $j=1,\dots,n$. Then
\begin{equation}
	\label{eq:BL_for_universal_Pi_j}
\dim Z\leq\frac{1}{2}\sum_{j=1}^n \dim Y_j.
\end{equation}

\begin{proof}
Let $S$ be the scheme-theoretic image of $Z\to \GL_n$ (see the diagram below). Let $K$ be an algebraic closure of the fraction field of $S$, so $\eta:\Spec K\to S$ is the geometric generic point of $S$. Consider the base change diagram below\footnote{$W$ and $Q_j$ will be introduced later; all other entries are obtained by base change. All monomorphism arrows are closed immersions.} --- all parallelograms are Cartesian. 

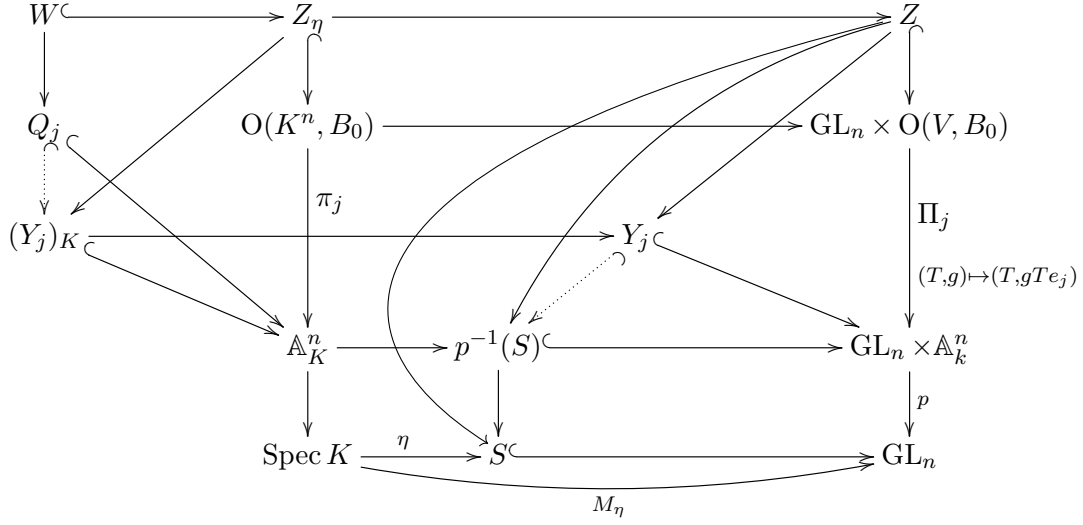
\[
\xymatrix{
	W\ar@{^{(}->}[rr]\ar[d] & {} & {\tikz[remember picture,baseline=(giantZeta.base)]\node[inner sep=0pt] (giantZeta) {$Z_\eta$};}\ar[rrrr]\ar@{^{(}->}[d]\ar[ddll] & {} & {} & {} & {\tikz[remember picture,baseline=(giantZ.base)]\node[inner sep=0pt] (giantZ) {$Z$};}\ar@{^{(}->}[d]\ar[ddll]\ar@/_2pc/[dddlll] \\
	Q_j\ar@{^{(}->}[rrdd]\ar@{^{(}.>}[d] & {} & \OO(K^n,B_0)\ar[rrrr]\ar[dd]^(.35){\textstyle \pi_j} & {} & {} & {} & {\tikz[remember picture,baseline=(giantGLOB.base)]\node[inner sep=0pt] (giantGLOB) {$\GL_n\times\OO(V,B_0)$};}\ar[dd]^(.41){\textstyle \Pi_j}^(.68){(T,g)\mapsto (T,gTe_j)} \\
	{\tikz[remember picture,baseline=(giantYjK.base)]\node[inner sep=0pt] (giantYjK) {$(Y_j)_K$};}\ar@{^{(}->}[rrd]\ar[rrrr] & {} & {} & {} & {\tikz[remember picture,baseline=(giantYj.base)]\node[inner sep=0pt] (giantYj) {$Y_j$};}\ar@{^{(}->}[rrd]\ar@{^{(}.>}[dl] & {} & {} \\
	{} & {} & {\tikz[remember picture,baseline=(giantAnK.base)]\node[inner sep=0pt] (giantAnK) {$\A^n_K$};}\ar[r]\ar[d] & {\tikz[remember picture,baseline=(giantPinvS.base)]\node[inner sep=0pt] (giantPinvS) {$p^{-1}(S)$};}\ar@{^{(}->}[rrr]\ar[d] & {} & {} & \GL_n\times\A^n_k\ar[d]^p \\
	{} & {} & \Spec K\ar[r]^\eta\ar@/_1pc/[rrrr]_{M_\eta} & {\tikz[remember picture,baseline=(giantS.base)]\node[inner sep=0pt] (giantS) {$S$};}\ar@{^{(}->}[rrr] & {} & {} & \GL_n
}
\]

\begin{tikzpicture}[remember picture,overlay]
	\coordinate (midYrow) at ($(giantYjK)!0.5!(giantYj)$);
	\coordinate (midArow) at ($(giantAnK)!0.5!(giantPinvS)$);
	\draw[->]
	($(giantZ.west)+(-0.2cm,-0.1cm)$)
	.. controls ($(midYrow)+(-0.2cm,1.0cm)$) and ($(midArow)+(-1.2cm,0.2cm)$) .. (giantS.north west);
\end{tikzpicture}

Set $d:=\dim S$, $r:=\dim Z_\eta$, and $r_j:=\dim (Y_j)_K$. Since $S\hookrightarrow \GL_n$ is a closed immersion, we have $Z_\eta\simeq Z\times_S \Spec K$. In particular, $Z_\eta$ is the (geometric) generic fiber of the dominant map $Z\to S$. Therefore 
\[\dim Z=d+r.\] 
The map $Y_j\hookrightarrow \GL_n\times\A^n_k$ factors through $p^{-1}(S)$, by the universal property of the scheme-theoretic image $Y_j$. Therefore, as in the case for $Z_\eta$ above, we have
$(Y_j)_K\simeq Y_j\times_S \Spec K$; thus $(Y_j)_K$ is the (geometric) generic fiber of the dominant map $Y_j\to S$, yielding
\[\dim Y_j=d+r_j.\]
Therefore \eqref{eq:BL_for_universal_Pi_j} is equivalent to $d+r\leq\frac{1}{2}\sum_{j=1}^n (d+r_j)$. Since $d\leq \frac{nd}{2}$, it suffices to establish
\begin{equation}
	\label{eq:BL_universal_family_suffices}
r\leq \frac{1}{2}\sum_{j=1}^n r_j.
\end{equation}

View $M_\eta$ (defined in the diagram) as an element in $\GL_n(K)$; write $\mathcal{u}_j\in K^n$ for the $j$-th column of $M_\eta$. After base change to $\Spec K$, the map $\Pi_j$ becomes $\pi_j:\OO(K^n,B_0)\to\A^n_K$, 
$g\mapsto g\mathcal{u}_j$. 

We apply Proposition \ref{Prop:OO_satisfies_Brascamp_Lieb} to the field $K$, the bilinear form on $K^n$ induced by $B_0$ (which we still denote by $B_0$), the basis $\mathcal{u}_1,\dots,\mathcal{u}_n$ of $K^n$, the maps $\pi_j$, and an irreducible component $W$ of $Z_\eta$ with $\dim W=r$.  

Let $Q_j$ denote the scheme-theoretic image of the composition $W\hookrightarrow\OO(K^n,B_0)\xrightarrow{\pi_j}\A^n_K$. By the universal property, the map $Q_j\hookrightarrow\A^n_K$ factors through $(Y_j)_K$. In particular, $\dim Q_j\leq r_j$. 

The first inequality below follows from Proposition \ref{Prop:OO_satisfies_Brascamp_Lieb}:
\[r\leq\frac{1}{2}\sum_{j=1}^n\dim Q_j\leq \frac{1}{2}\sum_{j=1}^n r_j,\]
establishing \eqref{eq:BL_universal_family_suffices}. 
\end{proof}

The claim implies that Conjecture \ref{Conj:Brascamp_Lieb} with $p_1=\dots=p_n=1/2$ applies to the maps $\Pi_j$, yielding a constant $C>0$ as in Conjecture \ref{Conj:Brascamp_Lieb}. Let $u_1,\dots,u_n$ be a basis of $V$. Let $T_0\in\GL_n(k)$ be the matrix whose $j$-th column is $u_j$. Take
\[A_j=\{T_0\}\times E\subset\GL_n(k)\times k^n;\quad\text{note that $|A_j|=|E|$}.\] 
There is a bijection 
\[\{ (T,g)\in \GL_n(k)\times\OO(V,B_0)(k)\ |\ \Pi_j(T,g)\in A_j\ \text{for all $j$} \}   \simeq  \{g\in \OO(V,B_0)(k)\ |\ gu_j\in E\ \text{for all $j$}\},\]
hence \eqref{eq:Brascamp_Lieb_conclusion} becomes \eqref{eq:O_n_goal_bound}. This establishes \eqref{eq:bound_cor_orthogonal}. 
\end{proof}

\subsection{The real orthogonal case}
\label{Sec:real_orthogonal_groups}

The following is a special case of \cite[Proposition 12]{BCELM}.

\begin{proposition}[Proposition 12 in \cite{BCELM}]
	Let $u_1,\dots,u_n$ be an orthonormal basis of $\R^n$ (with $n\geq 2$). Let $f_1,\dots,f_n:S^{n-1}\to \R$ be non-negative measurable functions. Let $\mu$ be the normalized bi-invariant Haar measure on $\SO(n)$. Then
	\begin{equation}\label{eq:BCELM_BL_Real_special_orthogonal}
		\int_{\SO(n)}\prod_{i=1}^n f_i(Uu_i)\ d\mu(U)\leq \prod_{i=1}^n \left(\int_{\SO(n)} f_i(Uu_i)^2\ d\mu(U)\right)^{1/2}.	
	\end{equation}
\end{proposition}

\begin{proof}
	In \cite[Proposition 12]{BCELM}, take $\cI=\{\{1\},\dots,\{n\}\}$; then $p=2$. Proposition 12 in \cite{BCELM} applies for any orthonormal basis $u_1,\dots,u_n$, not just the standard basis (see the paragraph preceding Proposition 11 in \cite{BCELM}). 
	Let $i\in\{1,\dots, n\}$, so $E_{\{i\}}=\spn(u_i)$. For $U\in\SO(n)$, the restriction $U|_{\spn(u_i)}$ is determined by $Uu_i\in S^{n-1}$; thus the space of restrictions $\{U|_{\spn(u_i)}: U\in \SO(n)\}$ is identified with $S^{n-1}$, which is the suitable space that \cite[Proposition 12]{BCELM} refers to in this particular case.   
\end{proof}

Let $\sigma$ be the normalized surface measure on $S^{n-1}$. Then for each $i$, the measure $\sigma$ coincides with the push-forward measure $(\pi_i)_*\mu$ of $\mu$ via $\pi_i:\SO(n)\to S^{n-1}$, $g\mapsto gu_i$; in other words, for each measurable set $B\subset S^{n-1}$, we have $\sigma(B)=\mu(\pi_i^{-1}(B))$. Then the integral appearing inside the $i$-th factor of the right-hand side of \eqref{eq:BCELM_BL_Real_special_orthogonal} becomes
$\int_{S^{n-1}}f_i(x)^2\ d\sigma(x)$. Therefore \eqref{eq:BCELM_BL_Real_special_orthogonal} becomes
\begin{equation}\label{eq:BCELM_BL_sphere_version}
	\int_{\SO(n)}\prod_{i=1}^n f_i(Uu_i)\ d\mu(U)\leq \prod_{i=1}^n \left(\int_{S^{n-1}} f_i(x)^2\ d\sigma(x)\right)^{1/2}.		
\end{equation}

Let $A\subset S^{n-1}$ be a measurable subset. Take each $f_i$ to be the indicator function of $A$, for $i=1,\dots,n$. Then \eqref{eq:BCELM_BL_sphere_version} gives
\begin{equation}\label{BCELM_measure_version}
	\mu\{U\in\SO(n)\ |\ Uu_i\in A \ \text{for each $i$}\}\leq \sigma(A)^{n/2}.	
\end{equation}

From here, a standard thickening argument yields the following unconditional case of Corollary \ref{Cor:General_case_orthogonal}. 

\begin{corollary}
	Let $E$ be a finite subset of $\R^n$. Let $u_1,\dots,u_n$ be an orthonormal basis of $\R^n$. Let $G$ be $\OO(n)$ or $\SO(n)$. Then
	\begin{equation}\label{eq:real_orthogonal_discrete_BL}
		\#\{U\in G(\R)\ |\ Uu_i\in E\ \text{for each $i$}\}\ll_n |E|^{n/2}.
	\end{equation}	
\end{corollary}

\begin{proof}
It suffices to treat the case $G=\SO(n)$. 	
	Since $Uu_i\in S^{n-1}$ for every $U\in\SO(n)$, replacing
	$E$ by $E\cap S^{n-1}$ leaves the set being counted unchanged and
	does not increase $|E|$. We may therefore assume that
	$E\subset S^{n-1}$. Also assume $n\geq 2$. 
	For $r>0$ and $x\in S^{n-1}$, define 
	\[C(x,r):=\{y\in S^{n-1}\ |\ |y-x|<r\}.\] 
	For sufficiently small $r$, consider $A_r:=\bigcup_{x\in E}C(x,r)$. For each $x\in E$, we have $\sigma(C(x,r))\ll_n r^{n-1}$ and consequently \begin{equation}\label{eq:spherical_cap}
		\sigma(A_r)\ll_n |E|r^{n-1}.
	\end{equation} 	
	
	Let $T=T^{\SO(n)}_{E,u}$ be the (finite) set whose cardinality appears on the left-hand side of \eqref{eq:real_orthogonal_discrete_BL}. For $U\in T$, define 
	\[B_r(U):=\{V\in \SO(n)\ |\ \lVert V-U\rVert_{\text{HS}}<r\},\]
	where $\lVert\cdot\rVert_{\text{HS}}$ is the Hilbert--Schmidt norm on matrices, induced from the Euclidean norm on $\R^{n^2}$. For sufficiently small $r>0$, the balls $\{B_r(U)\}_{U\in T}$ are disjoint and 
	\[\bigcup_{U\in T}B_r(U)\subset \{V\in\SO(n)\ |\ Vu_i\in A_r\ \text{for each $i$}\}.\]
	Each $\mu(B_r(U))\gg_n r^{\dim\SO(n)}=r^{n(n-1)/2}$. Consequently
	\begin{align*}
		|T|r^{n(n-1)/2} \ll_n\mu\left(\bigcup_{U\in T}B_r(U)\right) &\leq \mu \{V\in\SO(n)\ |\ Vu_i\in A_r\ \text{for each $i$}\}\\
		& \leq \sigma(A_r)^{n/2} &&\text{by \eqref{BCELM_measure_version}}\\ 
		&\ll_n (|E|r^{n-1})^{n/2} && \text{by \eqref{eq:spherical_cap}}\\
		&=|E|^{n/2}r^{n(n-1)/2}.
	\end{align*} 
	After canceling $r^{n(n-1)/2}$ on both sides, we obtain \eqref{eq:real_orthogonal_discrete_BL}. 
\end{proof}

\subsection{Alternative versions of the Brascamp--Lieb conjectural inequality}
\label{Sec:BL_Alternative_versions}

One way to eliminate the dependence on $k$ of the implied constant in \eqref{eq:bound_cor_orthogonal} is to state the following stronger version of Conjecture \ref{Conj:Brascamp_Lieb}. 

We say that an affine variety $X\subset\A^r$ has complexity $\leq M$ if $r\leq M$ and $X$ can be defined in $\A^r$ by $\leq M$ equations, each of degree $\leq M$. We say that a morphism $X\to Y$ of affine varieties $X\subset\A^r$ and $Y\subset \A^s$ (each of which has complexity $\leq M$) has complexity $\leq M$ if its graph 
in $\A^r\times\A^s$ has complexity $\leq M$. 

\begin{conjecture}
\label{Conj:BL_complexity}	
	Given $M>0$ and $p_1,\dots,p_m\in (0,1]$, there exists a $C>0$ with the following property. Let $k$ be a field. Let $X, Y_1, \dots, Y_m$ be  affine varieties over $k$ of complexity $\leq M$, and let $\pi_j: X\to Y_j$ be morphisms of complexity $\leq M$ (for $j=1,\dots,m$).
	Suppose that for every irreducible subvariety $Z\subset X_{\overline{k}}$,
	\begin{equation}
		\dim Z\leq \sum_{j=1}^m p_j\dim \overline{(\pi_{j})_{\overline{k}}(Z)}.
	\end{equation}
	Let $A_j\subset Y_j(k)$ be finite subsets. Then
	\begin{equation}
		\#\{x\in X(k)\ |\ \pi_j(x)\in A_j\text{ for all }j\}
		\leq C\prod_{j=1}^m |A_j|^{p_j}.
	\end{equation}
\end{conjecture}

This strong version (together with Proposition \ref{Prop:OO_satisfies_Brascamp_Lieb}) directly implies Corollary \ref{Cor:General_case_orthogonal} with $\ll_n$ in the estimate 
\eqref{eq:bound_cor_orthogonal}
and without the need for the argument in Section \ref{Sec:BL_family}.

Alternatively, Corollary \ref{Cor:General_case_orthogonal} would follow with $\ll_n$ in place of $\ll_{n,k}$ from the following version. 

\begin{conjecture}
	\label{Conj_BL_family}
Let $R_0$ be a ring, let $X, Y_1, \dots, Y_m$ be affine schemes of finite type over $R_0$, and let $\pi_j: X\to Y_j$ be morphisms ($j=1,\dots,m$). Let $p_1,\dots,p_m\in (0,1]$. Suppose that
for every field $k$ with a map $R_0\to k$,
\begin{equation*}
	\dim Z\leq \sum_{j=1}^m p_j\dim \overline{(\pi_{j})_{\overline{k}}(Z)}\quad\text{for every irreducible subvariety $Z\subset X_{\overline{k}}$}.
\end{equation*}
Then there exists a $C>0$ such that for every field $k$ with a map $R_0\to k$ and all finite subsets $A_j\subset Y_j(k)$, the following holds:
\begin{equation*}
	\#\{x\in X(k)\ |\ \pi_j(x)\in A_j\text{ for all }j\}
	\leq C\prod_{j=1}^m |A_j|^{p_j}.
\end{equation*}
\end{conjecture}

The reason this version implies Corollary \ref{Cor:General_case_orthogonal} with $\ll_n$ in \eqref{eq:bound_cor_orthogonal} is as follows. In the argument from Section \ref{Sec:BL_family}, take the standard form $B_0$.
The morphisms $\Pi_j$ are defined over $\mathbb Z$, hence in particular over $\mathbb Z[1/2]$. Taking $R_0=\mathbb Z[1/2]$, the Claim in the proof of Corollary \ref{Cor:General_case_orthogonal}
(based on Proposition \ref{Prop:OO_satisfies_Brascamp_Lieb}) verifies the dimension hypothesis for every field $k$ with a map $R_0\to k$.

However, we choose the version in Conjecture \ref{Conj:Brascamp_Lieb} because it is most intuitive and in line with the analogous statements in the literature in other categories. 

\section{Setwise stabilizers in characteristic zero}
\label{Sec:char_0}

We now prove Theorem \ref{Thm:bound_char_0_rank} by the following steps. 

\subsection{Reduction to the abelian case}

We apply Jordan's theorem: see \cite[Theorem~0.1]{LarsenPink} for a modern reference or \cite[p.~114]{Jordan1878} for the original source. 

\begin{theorem}[Jordan's theorem]
	\label{thm:jordan}
	For every positive integer $n$, there is a constant $J(n)$ such that if
	$k$ is a field of characteristic $0$ and $H\subset \GL_n(k)$ is finite,
	then $H$ contains an abelian normal subgroup $A$ with $[H:A]\leq J(n)$.
\end{theorem}

By \eqref{eq:comparison_R_E_T_E}, $R_E^G$ is finite. So Jordan's theorem gives an abelian $A\subset R_E^G$ whose index in $R_E^G$ is at most $J(n)$. Then 
\[|R_E^G|=[R_E^G:A]|A|\ll_n |A|.\]

Thus, to establish Theorem \ref{Thm:bound_char_0_rank}, it suffices to prove that $|A|\ll_n |E|^{\rank G}$ whenever $A$ is a finite abelian subgroup of $G(k)$ that preserves $E$ setwise. 

\subsection{Further reductions and preparations}
\label{Sec:char_0_preparations}

The case $G=\GL_n$ follows from the trivial bound \eqref{eq:trivial_bound} and the comparison \eqref{eq:comparison_R_E_T_E}. 

We may and shall assume that $k$ is algebraically closed. Since the cases with $n=1$ are trivial, we assume $n\geq 2$.

Let $\widehat{A}:=\Hom(A,k^*)$ be the character group of $A$. 

Our tool for giving a bound for $|A|$ is the following lemma. 

\begin{lemma}
	\label{lem:orbit_counting}
	Let $A$ be a finite abelian group acting linearly on a vector space $V$, and let
	$E\subset V$ be a finite subset which spans $V$. Suppose that $A$ stabilizes $E$ setwise.  Suppose there are
	characters $\lambda_1,\dots,\lambda_s\in\widehat{A}$ and nonzero linear forms
	$\ell_1,\dots,\ell_s\in V^*$ which satisfy
	\vspace*{-0.5cm}
	\[
	\ell_i(av)=\lambda_i(a)\ell_i(v)
	\qquad \text{for all $a\in A$, $v\in V$, and $1\leq i\leq s$.}
	\hfill\hspace*{3cm}
	\vcenter{\hbox{\Large$\xymatrix@=1.2em{
				V \ar[d]_{a} \ar[r]^{\ell_i} & k \ar[d]^{\lambda_i(a)} \\
				V \ar[r]^{\ell_i} & k
			}$}}
	\]
	\vspace*{-0.5cm}
	
	Let $K:=\bigcap\ker\lambda_i$. Then $|A|\leq |K||E|^s$. 
\end{lemma}

\begin{remark}
The case $K=\{1\}$ of Lemma \ref{lem:orbit_counting} stems from the standard base-counting argument from permutation group theory. 
Let a finite group $G$ act faithfully on a finite set $E$. As in \cite{Duyan_Halasi_Maroti_Adv}, a \defi{base} is a tuple
	$(e_1,\dots,e_s)$ of elements of $E$ such that
$\bigcap_{i=1}^s\operatorname{Stab}_G(e_i)=\{1\}$. The map $G\to E^s$, $g\mapsto (ge_1,\dots,ge_s)$ is injective, hence $|G|\leq |E|^s$.  More generally, one can
	consider a tuple $(e_1,\dots,e_s)$ such that $\bigcap_{i=1}^s\operatorname{Stab}_G(e_i)\subset K$, where $K$ is a fixed subgroup of $G$; see, for example,
	\cite[Section~2]{Duyan_Halasi_Maroti_Adv}. Then the orbit--stabilizer theorem applied to the diagonal action of $G$ on $E^s$ and the orbit of $(e_1,\dots,e_s)$ gives $|G|\leq |K|\,|E|^s$. The case of Lemma \ref{lem:orbit_counting} allowing a nontrivial $K$ is based on this idea. 
\end{remark}

\begin{proof}[Proof of Lemma \ref{lem:orbit_counting}]
	Since $E$ spans $V$ and $\ell_i\neq 0$, the linear form $\ell_i$ cannot
	vanish on every element of $E$. Hence, for each $i$, we can choose
	$v_i\in E$ with $\ell_i(v_i)\neq 0$.
	
	Consider the map
	\[
	\Phi:A\longrightarrow Av_1\times\cdots\times Av_s,
	\qquad
	\Phi(a)=(av_1,\dots,av_s).
	\]
	We prove that each fiber of $\Phi$ has size at most $|K|$. Suppose that $a,b\in A$ have
	the same image. Then
	$av_i=bv_i$
	for every $i$, so $c:=b^{-1}a$ fixes every $v_i$. Applying $\ell_i$ to
	$cv_i=v_i$, we get
	\[
	\lambda_i(c)\ell_i(v_i)=\ell_i(cv_i)=\ell_i(v_i).
	\]
	Because $\ell_i(v_i)\neq 0$, this implies $\lambda_i(c)=1$ for all $i$, hence $c\in K$.
	
	Each orbit $Av_i$ is contained in $E$, because $A$ acts on $E$ and $v_i\in E$.
	Therefore
	\[
	|A|\leq |K||\Phi(A)|\leq |K||Av_1|\cdots |Av_s|\leq |K||E|^s.\qedhere
	\]
\end{proof}

Since $A$ is a finite abelian group acting on the finite-dimensional vector space $V$ over an algebraically closed field of characteristic zero, $A$ is simultaneously diagonalizable. Let $e_1,\dots,e_n$ be a basis of $V$ consisting of simultaneous eigenvectors. Let $i\in\{1,\dots,n\}$. For each $a\in A$, we can write
\[ae_i=\chi_i(a)e_i\]
with a scalar $\chi_i(a)$. It is immediate to verify that $\chi_i:A\to k^*$ is a character of $A$.

\subsection{The case of $\SL_n$}

\begin{lemma}
	\label{lem:sl_bound}
	Let $k$ be an algebraically closed field of characteristic $0$, and let $V=k^n$. Let
	$A\subset \SL_n(k)$ be a finite abelian subgroup, and let 
	$E\subset V$ be
	a finite subset that spans $V$. Suppose that $E$ is preserved by $A$. Then
	\[
	|A|\leq |E|^{n-1}.
	\]
\end{lemma}

\begin{proof}[Proof of Lemma \ref{lem:sl_bound}]
	Consider a basis of simultaneous eigenvectors $e_1,\dots,e_n$ for $A$ and the corresponding characters $\chi_1,\dots,\chi_n\in\widehat{A}$ as in the end of Section \ref{Sec:char_0_preparations}. Let $e_1^*,\dots,e_n^*$ be the basis of $V^*$ dual to $e_1,\dots,e_n$. 
	
	We apply Lemma \ref{lem:orbit_counting} to the characters $\chi_1,\dots,\chi_{n-1}\in\widehat{A}$ and the linear forms $e_1^*,\dots,e_{n-1}^*\in V^*$. The statement of the lemma will follow once we check that $\bigcap_{i=1}^{n-1}\ker\chi_i=\{1\}$. Suppose that $a\in A$ fixes $e_1,\dots,e_{n-1}$. Then the determinant condition $\det(a)=1$ will force $a=1$:
	\[\det(e_1,\dots,e_n)=\det(ae_1,\dots, ae_n)=\det(e_1,\dots,e_{n-1},\chi_n(a)e_n)=\chi_n(a)\det(e_1,\dots,e_n),\]
	giving $\chi_n(a)=1$, hence $a$ fixes $e_n$ as well. Thus $a=1$. 	
\end{proof}

The preceding argument is simple because every finite abelian subgroup of
$\SL_n(k)$ is contained in a maximal torus. However, finite abelian subgroups of $\OO(V,B)$ or $\SO(V,B)$
need not be contained in a maximal torus. For example, in $\SO_3(k)$, the subgroup
of diagonal sign matrices of determinant $1$ is isomorphic to
$(\mathbb Z/2\mathbb Z)^2$, whereas a maximal torus in $\SO_3(k)$ has only one
nontrivial element of order $2$.
The next section develops a replacement argument using the pairing between
character spaces induced by the bilinear form.

\subsection{Character spaces and bilinear forms}

Let $k$ be an algebraically closed field of characteristic zero. 
Let $A$ be a finite abelian group acting on a finite-dimensional vector space $V$ over $k$. 
For a character $\chi\in\widehat{A}$, define the $\chi$-eigenspace as
\[
V_\chi=
\{v\in V\ |\ av=\chi(a)v \text{ for every } a\in A\}.
\]
It is an $A$-invariant subspace. Let $X=\{\chi\in\widehat{A}\ |\ V_\chi\neq \{0\}\}$. We have a direct sum decomposition
\begin{equation}
	V=\bigoplus_{\chi\in\widehat{A}} V_\chi=\bigoplus_{\chi\in X} V_\chi
	\label{eq:char_space_decomposition}	
\end{equation}

(this is well-known; see, for example, Theorem 9.3 in \cite{Conrad_FiniteAbelianCharacters}). The following elementary lemma is standard --- see, for example, the argument in \cite[Lemma 7.17]{MagaardMalleTiep} --- but we include a short proof for completeness.  

\begin{lemma}
	\label{lem:characters_forms}
	Let $A$ be a finite abelian group acting linearly on $V$, and let
	$\beta$ be a bilinear form on $V$ preserved by $A$. Let $X=\{\chi\in\widehat{A}\ |\ V_\chi\neq\{0\}\}.$
	\begin{itemize}
		\item[a)] If $u\in V_\chi$ and $v\in V_\psi$ with $\chi,\psi\in X$, then $\beta(u,v)=0$ unless $\chi\psi=1$. 
		\item[b)] Suppose that $\beta$ is non-degenerate. Let $\chi\in\widehat{A}$. Then $\beta$ restricts to a non-degenerate pairing
		$V_\chi\times V_{\chi^{-1}}\to k.$
		In particular, $\chi\in X$ implies $\chi^{-1}\in X$, and $\dim V_\chi=\dim V_{\chi^{-1}}$.
	\end{itemize}
\end{lemma}

\begin{proof}
	{\it a)} Let $u\in V_\chi$ and $v\in V_\psi$. For every $a\in A$, we have
	\[\beta(u,v)=\beta(au,av)=\beta(\chi(a)u,\psi(a)v)=\chi(a)\psi(a)\beta(u,v).\]
	
	Thus either $\beta(u,v)=0$, or else $\chi(a)\psi(a)=1$ for every $a\in A$, giving $\chi\psi=1$. 
	
	{\it b)} Suppose that $u\in V_\chi$ pairs
	trivially with every vector of $V_{\chi^{-1}}$. For every $\psi\neq\chi^{-1}$, part a) implies that $u$ pairs trivially with all vectors in $V_\psi$. But then the decomposition \eqref{eq:char_space_decomposition} implies that $u$ pairs trivially with all vectors of $V$. Since $\beta$ is non-degenerate, $u=0$. Similarly, only the vector $v=0$ in $V_{\chi^{-1}}$ pairs trivially with all vectors in $V_\chi$. 
\end{proof}

\subsection{The orthogonal and symplectic cases}

We treat the orthogonal and symplectic cases by a uniform argument.

\begin{proposition}
	\label{prop:form_preserving_bound}
	Let $k$ be an algebraically closed field of characteristic $0$. Let $V$ be
	an $n$-dimensional $k$-vector space equipped with a non-degenerate bilinear
	form $\beta$. Let $A\subset \GL(V)$ be a finite abelian subgroup which preserves $\beta$. Let $E\subset V$ be a finite subset that spans $V$. Suppose that $A$ preserves $E$ setwise. Then
	\[
	|A|\leq 2^n |E|^{\lfloor n/2\rfloor};\qquad\text{in particular,}\quad
	|A|\ll_n |E|^{\lfloor n/2\rfloor}.
	\]
\end{proposition}

\begin{proof}
	Let $X=\{\chi\in\widehat{A}\ |\ V_\chi\neq \{0\}\}$. 
	Consider the direct sum
	decomposition
	\begin{equation}\label{eq:char_0_direct_sum_V_proof_orth_sympl_gps}
		V=\bigoplus_{\chi\in X} V_\chi.
	\end{equation}
	Lemma \ref{lem:characters_forms} implies that the set $Y_0:=\{\chi\in X\ |\ \chi^2\neq 1\}$ splits into pairs $\{\chi,\chi^{-1}\}$, and $\dim V_\chi=\dim V_{\chi^{-1}}$ in each pair. Let $\chi_1,\dots,\chi_s$ be a list consisting of one representative from each pair. Since $V_{\chi_i}\neq\{0\}$ for each $i$, we have
	$s\leq \sum_{i=1}^s\dim V_{\chi_i}$. The splitting of $Y_0$ into pairs as above yields
	\[2s\leq \sum_{\chi\in Y_0}\dim V_{\chi}=\dim\left(\bigoplus_{\chi\in Y_0} V_\chi\right) \leq n,\]
	giving $s\leq \lfloor n/2\rfloor$. 
	
	For each $i$, let
	$\pi_{\chi_i}:V\to V_{\chi_i}$
	be the projection coming from the character-space 
decomposition.
	Choose a nonzero linear form
	$\ell_i:V_{\chi_i}\to k$.
	The diagram below commutes for each $a\in A$:
	\[
	\xymatrix{
		V \ar[d]^{a} \ar[r]^{\pi_{\chi_i}} & V_{\chi_i} \ar[d]^{a} \ar[r]^{\ell_i} & k \ar[d]^{\chi_i(a)} \\
		V \ar[r]^{\pi_{\chi_i}} & V_{\chi_i} \ar[r]^{\ell_i} & k
	}
	\]
	Therefore, we can apply Lemma \ref{lem:orbit_counting} for the characters $(\chi_i)_{i=1}^s$ and the linear forms $(\ell_i\circ \pi_{\chi_i})_{i=1}^s$: letting
	\[
	K:=\bigcap_{i=1}^s \ker \chi_i,
	\]
	it remains to prove that $|K|\leq 2^n$. 
	
{\bf Claim.} Let $Y:=\{\chi\in X\ |\ \chi^2=1\}$. Then the map $K\rightarrow \{\pm 1\}^{|Y|}$, $a\mapsto (\chi(a))_{\chi\in Y}$ is injective.

\begin{proof}[Proof of Claim]
Let $a\in K$ be such that $\chi(a)=1$ for every $\chi\in Y$. Then $a$ acts trivially on
\[
\bigoplus_{\chi\in Y} V_\chi.
\]
Moreover, since $a\in K$, by definition, $a$ acts trivially on each $V_{\chi_i}$. Then $a$ also acts trivially on $V_{\chi_i^{-1}}$, since $\chi_i^{-1}(a)=\chi_i(a)^{-1}=1$. So $a$ acts trivially on every $V_\chi$ with $\chi\in Y_0$. Then $a$ acts trivially on every summand
in the character-space decomposition \eqref{eq:char_0_direct_sum_V_proof_orth_sympl_gps} of $V$, and hence on all of $V$, forcing $a=1$.
\end{proof}	

The claim implies
	\[|K|\leq 2^{|Y|}\leq 2^{|X|}\leq 2^n\]
	(the bound $|X|\leq n$ on the last step follows from \eqref{eq:char_0_direct_sum_V_proof_orth_sympl_gps} and $V_\chi\neq \{0\}$ for every $\chi\in X$).
\end{proof}

Combining the reduction to the abelian case with Lemma \ref{lem:sl_bound}
and Proposition \ref{prop:form_preserving_bound} completes the proof of
Theorem \ref{Thm:bound_char_0_rank}.

\section*{AI Disclosure} 
This work was developed through extensive prompt interactions with 
OpenAI's ChatGPT 5.5 Pro model. 
The proofs in the manuscript originated from model-generated proposals. The model also suggested the statement of Proposition \ref{prop:weighted-general} and identified the Brascamp--Lieb analogy that led to the formulation of 
Conjectures \ref{Conj:Brascamp_Lieb}, \ref{Conj:BL_complexity}, and \ref{Conj_BL_family}. 
In this workflow, the author designed the project, formulated the statements to be proved (including the major results and, in some cases, smaller intermediate steps), 
directed the model through many rounds of prompting,
directed literature searches, assessed approaches, encouraged promising ideas, discarded unpromising ones, corrected mistakes, simplified arguments, and wrote up the final text with feedback from the model.  

\section*{Acknowledgements} 

I am grateful to Johannes Schmitt (ETH) for introducing me to AI and for sharing practical tips on effective prompting. I also thank Thang Pham and Le Quang-Hung for discussions related to the stabilizer problem in $\SL_n$ and our previous work \cite{Vietnam_BLMS}.


\begin{thebibliography}{99}
	
\bibitem{BCCT}
J. Bennett, A. Carbery, M. Christ, and T. Tao,
\emph{The Brascamp--Lieb inequalities: finiteness, structure and extremals},
Geom. Funct. Anal. \textbf{17} (2008), no.~5, 1343--1415.

\bibitem{BCCT2}
J. Bennett, A. Carbery, M. Christ, and T. Tao,
\emph{Finite bounds for H\"older--Brascamp--Lieb multilinear inequalities},
Math. Res. Lett. \textbf{17} (2010), no.~4, 647--666.

\bibitem{mult_linear_Brascamp_Lieb}
M. Christ, J. Demmel, N. Knight, T. Scanlon, and K. Yelick,
``On multilinear inequalities of H\"older--Brascamp--Lieb type for torsion-free discrete Abelian groups,''
\textit{Journal of Logic and Analysis} \textbf{16} (2024), Article 4.

\bibitem{Duncan}
J.~Duncan,
\emph{An algebraic Brascamp--Lieb inequality},
J. Geom. Anal. \textbf{31} (2021), 10136--10163.

\bibitem{Johnsrude}
B.~Johnsrude,
\emph{Discrete nonlinear H\"older--Brascamp--Lieb inequalities:
	a local approach},
arXiv:2607.21864, 2026.

\bibitem{KST}
T. K\H{o}v\'ari, V.~T. S\'os, and P. Tur\'an, 
On a problem of K. Zarankiewicz.
\textit{Colloquium Mathematicum}, \textbf{3} (1954), 50--57.

\bibitem{Yufei_Zhao}
Y.~Zhao,
\emph{Graph Theory and Additive Combinatorics:
	Exploring Structure and Randomness},
Cambridge University Press, Cambridge, 2023.
An author's version is available at
\url{https://yufeizhao.com/gtacbook/gtacbook.pdf}.

\bibitem{Do}
T. T. Do,
\emph{Zarankiewicz's problem for semi-algebraic hypergraphs},
J. Combin. Theory Ser. A \textbf{158} (2018), 621--642.
%doi:10.1016/j.jcta.2018.04.007.

	\bibitem{LoomisWhitney1949} 
L.~H. Loomis and H.~Whitney,
\emph{An inequality related to the isoperimetric inequality},
Bulletin of the American Mathematical Society \textbf{55} (1949), 961--962.

\bibitem{Finner1992}
H.~Finner,
\emph{A generalization of H\"older's inequality and some probability inequalities},
Ann. Probab. \textbf{20} (1992), no.~4, 1893--1901.

\bibitem{BCELM}
F.~Barthe, D.~Cordero-Erausquin, M.~Ledoux, and B.~Maurey,
\emph{Correlation and Brascamp--Lieb inequalities for Markov semigroups},
Int. Math. Res. Not. {\bf 2011} (2011), no.~10, 2177--2216.

\bibitem{pinned_distances}
B.~Murphy, G.~Petridis, T.~Pham, M.~Rudnev, and S.~Stevens,
\emph{On the pinned distances problem in positive characteristic},
J. Lond. Math. Soc. \textbf{105} (2022), no.~1, 469--499.

\bibitem{Vietnam_BLMS}
T. Pham, Le Quang-Hung, and K. Slavov, \emph{Sets preserved by a large subgroup of the special linear group}, Bull. London Math. Soc., {\bf 58} (2026): e70298.
%\url{https://doi.org/10.1112/blms.70298}

\bibitem{Pendavingh}
R.~Pendavingh,
\emph{Bounds on the number of cells and the dimension of the Dressian},
arXiv:2408.09466, 2024.

\bibitem{Jordan1878} 
C. Jordan,
\emph{M\'emoire sur les \'equations diff\'erentielles lin\'eaires \`a
	int\'egrale alg\'ebrique},
Journal f\"ur die reine und angewandte Mathematik \textbf{84} (1878),
89--215.

\bibitem{LarsenPink}
M.~J. Larsen and R.~Pink,
\emph{Finite subgroups of algebraic groups},
J. Amer. Math. Soc. \textbf{24} (2011), no.~4, 1105--1158.

\bibitem{Duyan_Halasi_Maroti_Adv}
H. Duyan, Z. Halasi and A. Maróti,
\emph{A proof of Pyber's base size conjecture},
Adv. Math. \textbf{331} (2018), 720--747.

\bibitem{Conrad_FiniteAbelianCharacters}
Keith Conrad,
\emph{Characters of finite abelian groups},
Theorem 9.3.
Available at \url{https://kconrad.math.uconn.edu/blurbs/grouptheory/charthy.pdf}.

\bibitem{MagaardMalleTiep}
K.~Magaard, G.~Malle, and P.~H.~Tiep,
\emph{Irreducibility of tensor squares, symmetric squares and alternating
	squares},
Pacific J. Math. \textbf{202} (2002), no.~2, 379--427.


\bibitem{Moonen}
B.~Moonen,
\emph{Introduction to Algebraic Geometry},
lecture notes based on the Mastermath course Algebraic Geometry, Spring 2013,
version of February 16, 2015.
Available at \url{https://www.math.ru.nl/~bmoonen/Lecturenotes/alggeom.pdf}.

\bibitem{StollLinearAlgebraII}
M.~Stoll,
\emph{Linear Algebra II},
course notes, with additions by R.~van Luijk.
Available at \url{https://pub.math.leidenuniv.nl/~luijkrmvan/ps/LinearAlgebra2.pdf}.
See Theorem~8.31.

\bibitem{CameronClassicalGroups}
P.~J. Cameron,
\emph{Classical Groups},
Chapter~6: Orthogonal groups.
Available at \url{https://webspace.maths.qmul.ac.uk/p.j.cameron/class_gps/ch6.pdf}.



\end{thebibliography}
\end{document}